\documentclass[11pt,a4paper,reqno, draft]{amsart}
\usepackage{amsmath, amssymb, enumerate,  graphicx,tikz, enumitem,MnSymbol}

\usepackage{thmtools, stmaryrd}

\usepackage[pdftex]{hyperref}
\usepackage{cleveref}
\usepackage{cases}
\usepackage{array}
\usepackage{tikz-cd}
\usepackage{yfonts}
\usepackage[all]{xy}
\usepackage{mathrsfs}
\usepackage[hmarginratio={1:1},vmarginratio={1:1},lmargin=80.0pt,tmargin=90.0pt]{geometry}

\usepackage{tikz}
\tikzset{anchorbase/.style={baseline={([yshift=-0.5ex]current bounding box.center)}}}
\usetikzlibrary{decorations.markings}
\usetikzlibrary{decorations.pathreplacing}
\usetikzlibrary{arrows,shapes,positioning,backgrounds}
\tikzstyle directed=[postaction={decorate,decoration={markings,
    mark=at position #1 with {\arrow{>}}}}]
\tikzstyle rdirected=[postaction={decorate,decoration={markings,
    mark=at position #1 with {\arrow{<}}}}]
    
\usepackage{graphicx}
\usepackage[vcentermath]{youngtab}
\usepackage{nicefrac}

\numberwithin{equation}{section}

\newtheorem{theorem}[subsubsection]{Theorem}
\newtheorem{lemma}[theorem]{Lemma}
\newtheorem{prop}[theorem]{Proposition}

\newtheorem{conjecture}[theorem]{Conjecture}

\theoremstyle{definition}
\newtheorem{definition}[subsubsection]{Definition}

\newtheorem{remark}[theorem]{Remark}
\newtheorem{problem}[theorem]{Problem}

\newtheorem{example}[subsubsection]{Example}

\newtheorem{question}[theorem]{Question}

\newcommand{\nilo}{\nil\hspace{-0.3mm}/G}
\newcommand{\bo}{\mathbf{0}}

\renewcommand{\H}{\mathcal{H}}

\newcommand{\nil}{\mathscr{N}}

\newcommand{\bI}{\mathbf{I}}
\newcommand{\bJ}{\mathbf{J}}
\newcommand{\bN}{\mathbf{N}}
\newcommand{\bP}{\mathbf{P}}
\newcommand{\bQ}{\mathbf{Q}}

\newcommand{\cD}{\mathcal{D}}

\newcommand{\co}{\mathrm{co}}
\newcommand{\ev}{\mathrm{ev}}
\newcommand{\bu}{\mathbf{u}}
\newcommand{\out}{\mathrm{out}}

\newcommand{\on}{\mathrm{on}}

\newcommand{\reg}{\mathrm{reg}}
\newcommand{\subreg}{\mathrm{subreg}}
\newcommand{\supmin}{{\mathrm{supmin}}}

\newcommand{\Rep}{\mathsf{Rep}}

\newcommand{\Tilt}{\mathsf{Tilt}}

\renewcommand{\a}{\mathsf{a}}
\newcommand{\sD}{\mathscr{D}}

\newcommand{\mW}{\mathbb{W}}

\newcommand{\cA}{\mathsf{A}}
\newcommand{\cR}{\mathcal{R}}

\newcommand{\cB}{\mathsf{B}}

\newcommand{\bk}{\mathbf{k}}

\newcommand{\tto}{\twoheadrightarrow}
\newcommand{\cO}{\mathscr{O}}
\newcommand{\cJ}{\mathcal{J}}

\newcommand{\mN}{\mathbb{N}}
\newcommand{\mG}{\mathbb{G}}
\newcommand{\mZ}{\mathbb{Z}}

\newcommand{\mC}{\mathbb{C}}
\newcommand{\mR}{\mathbb{R}}
\newcommand{\mP}{\mathbb{P}}
\newcommand{\mF}{\mathbb{F}}
\newcommand{\fg}{\mathfrak{g}}
\newcommand{\fl}{\mathfrak{l}}

\newcommand{\fh}{\mathfrak{h}}
\newcommand{\fn}{\mathfrak{n}}

\newcommand{\End}{\mathrm{End}}

\newcommand{\Ob}{\mathrm{Ob}}
\newcommand{\Ext}{\mathrm{Ext}}

\newcommand{\Hom}{\mathrm{Hom}}

\newcommand{\Sym}{\mathrm{Sym}}
\newcommand{\Rad}{\mathrm{Rad}}
\newcommand{\rad}{\mathrm{rad}}

\newcommand{\op}{\mathrm{op}}
\newcommand{\Ind}{\mathrm{Ind}}
\newcommand{\Res}{\mathrm{Res}}

\newcommand{\Id}{\mathrm{Id}}
\newcommand{\Vecc}{\mathsf{Vec}}

\newcommand{\Ver}{\mathsf{Ver}}

\newcommand{\zero}{{\mathrm{zero}}}

\renewcommand{\Vec}{\mathsf{Vec}}

\newcommand{\unit}{{\mathbf{1}}}
\newcommand{\cC}{\mathsf{C}}
\newcommand{\cQ}{\mathcal{Q}}
\newcommand{\cI}{\mathcal{I}}

\newcommand{\supp}{\mathrm{supp}}

\newcommand{\cS}{\mathsf{S}}

\newcommand{\incl}{\mathrm{incl}}

\newcommand{\mX}{\mathbb{X}}

\newcommand{\cN}{\mathcal{N}}

\newcommand{\FPdim}{\mathrm{FPdim}}

\newcommand{\id}{\mathrm{id}}

\newcommand{\Tr}{\operatorname{Tr}}

\newcommand{\tA}{\mathsf{A}}

\newcommand{\tB}{\mathsf{B}}
\newcommand{\tG}{\mathsf{G}}
\newcommand{\asph}{\mathrm{asph}}

\newcommand{\OB}{\mathcal{O}\mathcal{B}}

\begin{document}
\title[Tensor ideals of abelian type]{Tensor ideals of abelian type and quantum groups}
\author{Kevin Coulembier}
\address{K.C.: School of Mathematics and Statistics, University of Sydney, NSW 2006, Australia}
\email{kevin.coulembier@sydney.edu.au}

\author{Pavel Etingof}

\address{P.E.: Department of Mathematics, MIT, Cambridge, MA USA 02139}
\email{etingof@math.mit.edu}

\author{Victor Ostrik}
\address{V.O.: Department of Mathematics,
University of Oregon,
Eugene, OR USA 97403}
\email{vostrik@uoregon.edu}


\subjclass[2020]{}


\dedicatory{Dedicated to the memory of J\'er\'emie Guilhot}

\begin{abstract}
We initiate a study of tensor ideals in linear rigid monoidal categories that are kernels of linear monoidal functors to abelian monoidal categories. We develop general methods and apply them to the category of tilting modules over quantum groups as well as to some representation categories of finite groups. In an appendix on Duflo involutions in monoidal categories, we make a connection between Duflo involutions in the affine Weyl group and tensor ideals for quantum groups, and prove some of Lusztig's conjectures for arbitrary Coxeter groups, at equal parameters, without invoking the boundedness hypothesis.
\end{abstract}

\maketitle

\section*{Introduction}
This paper is motivated by recent developments in the theory of symmetric tensor categories, meaning $\bk$-linear abelian symmetric rigid monoidal categories over a field $\bk$ with some finiteness assumptions. If $\bk$ has characteristic 0, then classical work by Deligne shows that all such categories that are of moderate growth are representation categories of supergroups, that is affine group schemes in the symmetric tensor category of supervector spaces. Over a field $\bk$ of characteristic $p>0$ it was conjectured in \cite{BEO} that similarly all symmetric tensor categories of moderate growth are representation categories over a symmetric tensor category $\Ver_{p^\infty}=\cup_n\Ver_{p^n}$. The case `$n=1$' of this conjecture was proved in \cite{CEO}. 

The categories $\Ver_{p^n}$, for $n$ a positive integer, were constructed in \cite{BEO, AbEnv} as `abelian envelopes' of quotients of the symmetric monoidal category $\Tilt SL_2$ by its non-zero tensor ideals (which are labelled by positive integers). A natural first test of the conjecture in \cite{BEO} would be constructing abelian envelopes (if they exist) of quotients of $\Tilt G$ for other reductive groups $G$, and subsequently verifying whether indeed they fibre over $\Ver_{p^n}$ for some $n$. The first major obstacle is that a classification of the tensor ideals in $\Tilt G$ seems intractable at present, for $G$ not of rank 1. On the other hand, as we will demonstrate in the current paper, most tensor ideals in $\Tilt G$ will lead to quotients that do not admit faithful functors to tensor categories, and are thus irrelevant for the above goals. 

We therefore define a tensor ideal in a $\bk$-linear rigid monoidal category over a field $\bk$ to be {\em of abelian type} if it is the kernel of a $\bk$-linear monoidal functor to a tensor category over~$\bk$. We develop methods to classify tensor ideals of abelian type. We will pay specific attention to the category of tilting modules $\Tilt U_\zeta(\fg)$ of the quantum group at a root $\zeta$ of $1$ for $\fg$ a reductive Lie algebra over the complex numbers, because it often behaves as a `first order approximation' to $\Tilt G$. 

An obvious but important observation is that tensor ideals of abelian type must be `prime', see \ref{defideals}. The latter, being an intrinsic property of a tensor ideal, is often easier to handle. Several of our results are thus aimed at demonstrating that certain prime tensor ideals must be of abelian type, see for instance Theorem~\ref{ThmAbEnv}, and that the majority of tensor ideals are not prime, see for instance Theorem~\ref{Thm:FinR}.

A coarser notion of tensor ideal to classify is that of a `thick tensor ideal', which is a collection of objects that form an ideal under tensor product, and are closed under taking direct sums and summands. Indeed, while the classification of {\em thick} tensor ideals in $\Tilt U_\zeta(\fg)$ has been understood for nearly 30 years, see \cite{Os}, the full classification of all tensor ideals might well be intractable for $\fg$ of rank 3 or higher, as is corroborated from the complexity in Proposition~\ref{prop:B2}. Interestingly, while the number of tensor ideals compared to thick tensor ideals tends to be infinite to one, the assignment from prime tensor ideals to prime thick tensor ideals tends to be much closer to a bijection.

Our main strategy for understanding tensor ideals of abelian type is thus classifying prime tensor ideals based on a classification of prime thick tensor ideals and subsequently determining which prime tensor ideals are of abelian type. We demonstrate this approach with classifications of tensor ideals of abelian type in the representation categories of cyclic groups of order a power of the characteristic of the field, as well as in $\Tilt U_\zeta(\fg)$. Since the former case behaves as a simplified version of the latter, here we focus on quantum groups.

Assuming the naturality conjecture for the support of tilting modules, see \S\ref{sec:nat}, which is valid in type A, this leads to a classification of the tensor ideals of abelian type.
 
 \medskip
 
 {\bf Main Theorem:} {\em Under the above hypotheses and the conditions in \ref{sec:condl}, there is a canonical bijection between the tensor ideals of abelian type in $\Tilt U_\zeta(\fg)$ and the nilpotent orbits in $\fg$. Every prime tensor ideal in $\Tilt U_\zeta(\fg)$ is of abelian type.}
 
 \medskip
 
The tensor categories that we construct in order to establish that the prime tensor ideals are of abelian type, are the abelian envelopes of quotients of $\Tilt U_\zeta(\fl)$, corresponding to distinguished nilpotent orbits, for Levi subalgebras $\fl\subset\fg$.  For details, we refer to Theorems~\ref{ThmClassPrime}, \ref{thm:distab} and~\ref{ThmAb}.

 The paper is organised as follows. In Section~\ref{sec:prel} we recall some background. In Section~\ref{sec:abtype} we develop general theory of tensor ideals of abelian type and prime tensor ideals. In Section~\ref{sec:fingrp} we apply our techniques to the representation categories of some finite groups. In Section~\ref{sec:quantum1} we make some preparations for applying our methods to $\Tilt U_\zeta(\fg)$. We mainly recall some standard background, but also establish the apparently little-known fact that antispherical Kazhdan-Lusztig polynomials for an affine Weyl group are bounded in degree by the length of the longest element in the finite Weyl group. In Section~\ref{sec:quantum2}, we classify prime tensor ideals in $\Tilt U_\zeta(\fg)$ unconditionally, and classify the tensor ideals of abelian type assuming the naturality conjecture. In Section~\ref{sec:quantum3} we illustrate this classification by discussing the classification of {\em all} tensor ideals for quantum groups of rank 2. In Section~\ref{sec:further} we briefly discuss the problem of classification of tensor ideals in the oriented Brauer category in positive characteristic and in the category of tilting modules for reductive groups. Finally, in Appendix~\ref{sec:Duflo-Mon}, we establish a theory of Duflo involutions in specific rigid monoidal categories. Applying the results to the category of Soergel bimodules recovers the usual Duflo involutions, and more generally the resulting objects are always Duflo objects in the sense of \cite{MM, MM2}. In this way we manage to prove some of the basic facts on Duflo involutions (including their existence) for arbitrary Coxeter groups, without using Lusztig's boundedness hypothesis from \cite{Lubook, bounded}, but relying on the categorification results from \cite{EW14, So}. In relation to the main aims of the paper, we use these results to illuminate, in Section~\ref{sec:quantum2}, the precise connection between Duflo involutions and the classification of thick tensor ideals in $\Tilt U_\zeta(\fg)$.

\section{Preliminaries}\label{sec:prel}

Let $\bk$ be an arbitrary field, unless further specified, and write $\mN=\mZ_{\ge 0}$.

\subsection{Pseudo-tensor categories} We refer to \cite{EGNO} for details on monoidal categories.

\subsubsection{}\label{def:pseudo}
An essentially small $\bk$-linear monoidal category $(\cA,\otimes,\unit)$ is a {\bf pseudo-tensor category over $\bk$} if
\begin{enumerate}
\item $\cA$ is pseudo-abelian (additive and idempotent closed) and has finite dimensional morphism spaces;
\item $\bk\to\End(\unit)$ is an isomorphism;
\item $(\cA,\otimes,\unit)$ is rigid, meaning every object $X\in\cA$ has a left and right dual.
\end{enumerate}

Recall that a {\bf left dual} of $X\in\cA$ is a triple $(X^\ast,\ev_X,\co_X)$ of an object $X^\ast$ and morphisms $\ev_X:X^\ast\otimes X\to\unit$ and $\co_X: \unit\to X\otimes X^\ast$ satisfying $\id_X=(X\otimes \ev_X)\circ (\co_X\otimes X)$ and $\id_{X^\ast}=(\ev_X\otimes X^\ast)\circ (X^\ast\otimes \co_X)$. A right dual ${}^\ast X$ of $X$ is similarly defined, so that $X$ corresponds to a left dual of ${}^\ast X$. For a morphism $f:X\to Y$, the (co)evaluations define a morphism $f^\ast:Y^\ast\to X^\ast$. We use notation like $X^{(n)}$ and $f^{(n)}$, with $n\in\mZ$, for iterated duals, for instance $X^{(-2)}={}^{\ast\ast}X$.

A $\bk$-linear monoidal functor between pseudo-tensor categories will be called a {\bf tensor functor}. 
A bi-natural isomorphism $X\otimes Y\xrightarrow{\sim} Y\otimes X$ in a pseudo-tensor category will be called a {\bf commutor}. Commutors can satisfy additional properties that make the monoidal categories braided (e.g. symmetric) or coboundary, see \cite{HK}. We will often just require the existence of some commutor, and then speak of pseudo-tensor categories `with commutor'.

\subsubsection{} For a morphism $a:X\to X^{\ast\ast}$ in a pseudo-tensor category, following \cite[\S 4.7]{EGNO}, its left quantum trace is
$$\Tr^L(a)\;=\; \ev_{X^\ast}\circ(a\otimes  X^\ast)\circ \co_X\;\in\;\bk=\End(\unit).$$ Given a natural isomorphism $\psi:\Id\Rightarrow \Id^{\ast\ast}$ (not necessarily monoidal), we define the trace of an endomorphism $f\in \End(X)$ as
$$\Tr(f)\;=\;\Tr_{\psi}(f)\;:=\;\Tr^L(\psi_X\circ f).$$
The categorical dimension $\dim(X)=\dim_{\psi}(X)\in \bk$ of $X\in\cA$ is then $\Tr(\id_X)$.

\subsubsection{}\label{def:tensorcat}A pseudo-tensor category is a {\bf tensor category} if condition (1) is replaced with the condition that $\cA$ be abelian with all objects of finite length (which implies finite dimensionality of morphisms spaces). Moreover, in a tensor category $\unit$ is simple, see e.g. \cite[Theorem~4.3.8]{EGNO}. The tensor category of finite dimensional vector spaces is denoted by $\Vecc$.

A $\bk$-linear monoidal category $(\cA,\otimes,\unit)$ that satisfies (2) and (3) can be formally completed into its {\bf pseudo-abelian envelope}, which is a pseudo-tensor category, see \cite[\S 1.2]{AK}.

There is no uniform procedure to produce an abelian envelope of a pseudo-tensor category. However, the general theory of \cite{HomAb} establishes that for every pseudo-tensor category $\cA$, there exists a family $\{\cA\to\cC_\alpha\mid\alpha\in S\}$ of faithful tensor functors into tensor categories (the local abelian envelopes of $\cA$) such that for any faithful tensor functor $\cA\to\cC$ to a tensor category $\cC$ there are unique $\alpha\in S$ and exact tensor functor $\cC_\alpha\to\cC$ so that $\cA\to\cC$ factors as $\cA\to\cC_\alpha\to\cC$. If $S$, which can be empty, finite or infinite, is a singleton then the unique $\cA\to\cC_\alpha=:\cC$ is called the {\bf abelian envelope} of $\cA$. Sometimes the term abelian envelope is only applied if the functor $\cA\to\cC$ is also full.

The following is \cite[Corollary~4.4.4]{pretop}.
\begin{lemma}\label{LemGF}
Let $\Phi:\cA\to\cC$ be a faithful tensor functor from a pseudo-tensor category $\cA$ to a tensor category $\cC$. If
\begin{enumerate}
\item[(G)] every object in $\cC$ is a quotient of one in the essential image of $\Phi$, and
\item[(F)] for every morphism $a: \Phi(X)\to \Phi(Y)$ in $\cC$, there exists a morphism $q\in \cA(X',X)$ for which $\Phi(q)$ is an epimorphism and $a\circ \Phi(q)$ is in the image of $\Phi$,
\end{enumerate}
then $\Phi$ is an abelian envelope.
\end{lemma}

\subsection{Tensor ideals}
Let $\cA$ be a pseudo-tensor category.

\subsubsection{}\label{defideals} A {\bf thick tensor ideal $\bI$} in $\cA$ is a collection of objects (closed under isomorphism) in $\cA$ such that $X\in\bI$ implies that $Y\otimes X\in \bI\ni X\otimes Y$, and furthermore $X_1\oplus X_2\in\bI$ if and only if $X_1$ and $X_2$ are both in $\bI$. A thick tensor ideal $\bI$ is {\bf prime} if it does not contain $\unit$ and if $X\otimes Y\in\bI$ implies that at least one of $X,Y$ is in $\bI$. 

A {\bf tensor ideal} $\cI$ in $\cA$ is an ideal in the additive sense (a system of subspaces $\cI(X,Y)\subset\Hom(X,Y)$ for all $X,Y\in \cA$ that is closed under composition on either side with arbitrary morphisms) such that $f\in \cI$ implies that $Z\otimes f\in\cI\ni f\otimes Z$ for all $Z\in\cA$. A tensor ideal $\cI$ is {\bf prime}\footnote{In the case of pseudo-tensor categories without commutor, the term `completely prime' would be more appropriate, and a straightforward adaptation of our definition should be called prime. Especially if one drops the requirement $\End(\unit)=\bk$, the current definition would be problematic, since the zero ideal in a general multi-tensor category would not be prime.} if it is not the identity ideal and if $f\otimes g\in\cI$ implies that at least one of $f,g$ is in $\cI$.  A tensor ideal $\cI$ is {\bf faithfully prime} if the preimage of $\cI(X\otimes U,Y\otimes V)$ under
$$\Hom(X,Y)\otimes_{\bk}\Hom(U,V)\;\xrightarrow{\otimes}\;\Hom(X\otimes U,Y\otimes V)$$
is given by 
$$\cI(X,Y)\otimes_{\bk}\Hom(U,V)+\Hom(X,Y)\otimes_{\bk}\cI(U,V).$$
This is the same as requiring that the tensor product on the quotient category is faithful as a $\bk$-linear functor $\cA/\cI\,\otimes_{\bk}\,\cA/\cI\to \cA/\cI.$ Clearly faithfully prime ideals are prime.

Examples of prime tensor ideals that are not faithfully prime are given in Section~\ref{Klein}. Up to formulation, the following lemma (with proof) is contained in \cite[Lemmas~3.1 and~3.2]{OS}:

\begin{lemma} Assume that $\cA$ has a commutor.
  The following conditions are equivalent on a tensor ideal $\cI$ in $\cA$:
  \begin{enumerate}
  \item $\cI$ is faithfully prime;
  \item $\cI$ is prime, and for all morphisms $f,g:X\rightrightarrows Y$ in $\cA$, the property $f\otimes g-g\otimes f\in \cI$ implies that a non-trivial linear combination of $f$ and $g$ is in $\cI$.
  \end{enumerate}
 \end{lemma}
 \begin{proof}
 That (1) implies (2) is clear. To show that (2) implies (1), without loss of generality we can focus on the case $\cI=0$.  We then need to show that if
 $\sum_{i=1}^n f'_i\otimes f_i=0$
 for morphisms $f'_i:X'\to Y'$ and linearly independent morphisms $f_i:X\to Y$, then $f'_i=0$ for all $i$ (and the analogous claim for $f_i'$ linearly independent). We do this by induction on $n$, where the case $n=1$ is true by assumption ($0$ is prime).  The assumption $\sum_{i=1}^n f'_i\otimes f_i=0$ implies
 $$\sum_{i=1}^n f'_i\otimes f_i\otimes f_n=0=\sum_{i=1}^n f'_i\otimes f_n\otimes f_i,$$
 using the commutor.
 Subtracting both equalities shows
$$ \sum_{i=1}^{n-1} f'_i\otimes (f_i\otimes f_n-f_n\otimes f_i)\;=\;0.$$ By our assumptions, the morphisms $f_i\otimes f_n-f_n\otimes f_i$ are linearly independent, for $1\le i<n$. The induction hypothesis thus implies that $f_i'=0$ for $i<n$, and that $f'_n=0$ then follows similarly, or from the prime condition.
 \end{proof}

\subsubsection{}For a tensor ideal $\cI$, we let $\Ob(\cI)$ denote the thick tensor ideal of objects $X$ with $\id_X\in\cI$. These are the objects that become isomorphic to zero in the pseudo-tensor category $\cA/\cI$. We thus obtain a map
\begin{equation}\label{ObAll}\{\mbox{tensor ideals}\}\;\;\xrightarrow{\hspace{4mm}\Ob\hspace{4mm}}\;\;\{\mbox{thick tensor ideals}\}.\end{equation}
In \cite{Selecta}, several examples where \eqref{ObAll}
is a bijection were described. Despite this collection of natural examples, this map will rarely be a bijection in general, see for instance Lemma~\ref{lem:cyclic} and Section~\ref{sec:quantum3} below. However, the map is always surjective, and each fibre has a unique maximal and minimal element, see \S\ref{sec-fibre}. 

There are obvious left and right versions of tensor ideals, which we will always let refer to the tensor product alone (not composition), and of thick tensor ideals. For pseudo-tensor categories with commutor, left and right tensor ideals are simply tensor ideals. Indeed, if $\{\sigma_{A,B}:A\otimes B\xrightarrow{\sim}B\otimes A\}$ is a commutor for $\cA$, then for any $f:X\to Y$ and $Z$ in $\cA$, we have
$$Z\otimes f\;=\; \sigma_{Y,Z}^{-1}\circ (f\otimes Z)\circ \sigma_{X,Z}.$$
The following is \cite[Theorem~3.1.1]{Selecta} and goes back to \cite{AK}.

\begin{theorem}\label{ThmRigBij}
The assignment $\cJ\mapsto \cJ(\unit,-)$ yields an isomorphism between the poset of left (resp. right) tensor ideals in $\cA$ and the poset of subfunctors of $\cA(\unit,-):\cA\to\Vec$. The inverse map sends a subfunctor $J\subset\cA(\unit,-)$ to the left (resp. right) tensor ideal containing all morphisms $f:X\to Y$ for which the corresponding $\unit\to {}^\ast X\otimes Y$ (resp. $\unit\to Y\otimes X^\ast$) is in~$J({}^\ast X\otimes Y)$ (resp. $J(Y\otimes X^\ast)$).
\end{theorem}

\subsection{Negligible morphisms}

The results in this section go back to \cite{AK, BW}. 

\subsubsection{}\label{def:Jac} For a $\bk$-linear category $\cB$, its (Jacobson) radical $\cR=\cR(\cB)$ is a two-sided ideal, with
$$\cR(X,Y)\;=\;\{f\in \cB(X,Y)\mid \id_X-g\circ f\;\;\mbox{ invertible for all }\;\;g\in \cB(Y,X) \}.$$

\begin{prop}
Let $\cA$ be a pseudo-tensor category. Then $\cA$ has a unique maximal left tensor ideal $\cN_L$, and a unique maximal right tensor ideal $\cN_R$. Concretely:
$$\cN_R(X,Y)\;=\;\{f\in \cA(X,Y)\mid  \Tr^L(g\circ f)=0\mbox{ for all }g\in \cA(Y,X^{\ast\ast})\}.$$
\end{prop}
\begin{proof}
By Theorem~\ref{ThmRigBij}, it suffices to observe that $\cA(\unit,-)$ has a unique maximal proper subfunctor $\cR(\unit,-)$. Concretely, $\cR(\unit, X)=\cA(\unit, X)$ for all indecomposable $X\not=\unit$ and $\cR(\unit,\unit)=0$.
The displayed equation for $\cN_R$ can be verified directly. Alternatively, we can observe directly that $\cN_R$ is a right tensor ideal, after which it suffices to observe that $\cN_R(\unit,X)=\cR(\unit,X)$.
\end{proof}

Hence, if $\cA$ has a commutor, there is a unique maximal tensor ideal $\cN=\cN_L=\cN_R$, the ideal of {\bf negligible morphisms}. Additionally, for a natural isomorphism $\psi:\Id\Rightarrow \Id^{\ast\ast}$, which always exists if $\cA$ is braided, see \cite[8.10.6]{EGNO}, we can express $\cN$ in terms of $\Tr_{\psi}$. Also for a pivotal structure (without braiding requirement) there is a similar conclusion, see \cite[Theorem~2.6]{EO}:

\begin{theorem}\label{ThmN}
Let $\cA$ be a pseudo-tensor category with a pivotal structure $a$, see \cite[\S 4.7]{EGNO}. Then $\cN=\cN_R=\cN_L$ is the two-sided tensor ideal
$$\cN(X,Y)\;=\;\{f\in \cA(X,Y)\mid  \Tr_{a}(g\circ f)=0\mbox{ for all }g\in \cA(Y,X)\}.$$
If furthermore, for any nilpotent endomorphism $f$ in $\cA$ we have $\Tr_a(f)=0$, and for any indecomposable $X\in\cA$ we have $\dim_a(X)=0$ if and only if $\dim_a(X^\ast)=0$, then $\overline{\cA}:=\cA/\cN$ is a semisimple tensor category.
\end{theorem}

We will only apply Theorem~\ref{ThmN} in case $\cA$ is both braided and pivotal (equivalently $\cA$ is braided with a twist, e.g. a ribbon monoidal category, see \cite[\S 8.10]{EGNO}). Note that the nilpotence condition in Theorem~\ref{ThmN} is satisfied if $\cA$ admits any tensor functor to a tensor category, see Proposition~\ref{propnotab2} below. Besides the notation $\overline{\cA}$ for this {\bf semisimplification}, we denote the corresponding quotient functor by
$$\cA\;\to\;\overline{\cA}=\cA/\cN,\quad X\mapsto \overline{X}.$$

\begin{remark}\label{RemNegSingle}
Let $\cA$ be a pivotal pseudo-tensor category as in Theorem~\ref{ThmN}. 
The thick ideal of negligible objects is $\bN:=\Ob(\cN)$, and an indecomposable $X\in\cA$ is in $\bN$ if and only if $\dim(X)=0$, see \cite[Theorem~2.6]{EO}.
If it is furthermore the case that for all indecomposable $X,Y\not\in \bN$ in $\cA$, we have
$$\cA(X,Y)=\begin{cases}\bk,&\mbox{if $X\simeq Y$}\\
0,&\mbox{otherwise},\end{cases}$$
 then $\Ob^{-1}(\bN)=\{\cN\}$. This is for instance the case for $\cA$ the category of tilting modules of quantum groups, see Section~\ref{sec:quantum1}.
\end{remark}

\begin{remark}\label{rem:dic}
One can view the theory of tensor ideals as a categorification of the theory of ideals in rings. In this analogy, multi-pseudo-tensor categories (dropping condition (2) in \ref{def:pseudo}) correspond to rings, and Theorem~\ref{ThmRigBij} shows that pseudo-tensor categories with commutor correspond to commutative local rings. Prime tensor ideals correspond to (completely) prime ideals in rings. Depending on the point of view, the analogue of fields would be semisimple tensor categories with commutor, or more generally pseudo-tensor categories with commutor with only the zero ideal as proper tensor ideal. The role of the Jacobson radical should not be played by $\cR$ from \ref{def:Jac}, as it does not refer to the tensor product. In the philosophy of the current paper, the role of the Jacobson radical will be played by the unique maximal tensor ideal contained in the radical $\cR$, which exists by \ref{minmax}. Other choices also make sense, for instance, one could also declare the analogue of the Jacobson radical in a pseudo-tensor category with commutor to be the unique maximal ideal.
\end{remark}

\section{Tensor ideals of abelian type}\label{sec:abtype}

Unless further specified, $\cA$ is an arbitrary pseudo-tensor category over $\bk$. 

\subsection{Definition and recognition}

\begin{definition}\label{DefAbT}
A tensor ideal in a pseudo-tensor category $\cA$ over $\bk$ is of {\bf abelian type} if it is the kernel of a tensor functor to a tensor category over $\bk$.
\end{definition}

\begin{example}\label{preimN}
Let $\cA$ be a pseudo-tensor category, with a tensor functor to a pivotal pseudo-tensor category $\cB$ satisfying the assumptions in \ref{ThmN}. Then the pre-image in $\cA$ of $\cN(\cB)$ is a tensor ideal of abelian type. 
\end{example}

We have the following two general principles at our disposal to demonstrate that a given tensor ideal is {\bf not} of abelian type.

\begin{prop}\label{propnotab1}
If a tensor ideal $I$ in $\cA$ is of abelian type, then it is faithfully prime.
\end{prop}
\begin{proof}
It suffices to observe the (well-known) property that if $\cA$ is a tensor category, then
$$\cA(X,Y)\otimes_{\bk}\cA(U,V)\;\xrightarrow{\otimes}\;\cA(X\otimes U, Y\otimes V)$$
is injective.
By adjunction, we can assume that $Y=\unit=U$. With $x,v$ the dimensions of $\cA(X,\unit)$ and $\cA(\unit, V)$, we can then reduce to the case $\cA=\Vecc$, by replacing $X$ with its quotient $\unit^x$ and $V$ with its subobject $\unit^v$.
\end{proof}


\begin{prop}\label{propnotab2}
If $\cA$ has any tensor ideal of abelian type, then for any morphism of the form $f:X\to X^{\ast\ast}$ satisfying $f^{(2n)}\circ \cdots\circ f^{\ast\ast}\circ f=0$, for some $n\in \mN$, it follows that $\Tr^L(f)=0$.
\end{prop}
\begin{proof}
This result is well-known in the symmetric setting, see~\cite{De}. In general, it suffices to prove that the condition is satisfied if $\cA$ is a tensor category. Consider then the kernel $W\subset X$ of $f$. Since $X$ has the same length as $X^{\ast\ast}$, it follows that $W\not=0$. By \cite[Proposition~4.7.5]{EGNO}, it follows that 
$\Tr^L(f)=\Tr^L(g)$
where $g$ is the morphism $X/W\to X^{\ast\ast}/W^{\ast\ast}$ induced from $f$. Since $g^{(2n-2)}\circ \cdots\circ g^{\ast\ast}\circ g=0$, we can prove the claim by induction on $n$, with the base case $n=0$ being trivial.
\end{proof}

We will say that a pseudo-tensor category $\cA$ is {\bf trace-abelian} if $\Tr^L(f)=0$ for every $f:X\to X^{\ast\ast}$ in $\cA$ with $f^{(2n)}\circ \cdots\circ f^{\ast\ast}\circ f=0$ for some $n\in \mN$. If there is a natural isomorphism $\psi:\Id\Rightarrow\Id^{\ast\ast}$, the category $\cA$ is trace-abelian if and only if $\Tr_{\psi}(f)=0$ for all nilpotent $f$.

\begin{remark}\label{RemNecCond}
\begin{enumerate}
\item Another necessary condition on a tensor ideal $\cI$ to be of abelian type is that $(\cA/\cI)(\unit,L)$ is at most one-dimensional when $L$ is invertible.
\item We know no examples of faithfully prime tensor ideals in trace-abelian pseudo-tensor categories $\cA$ that are not of abelian type. However, there are prime tensor ideals in trace-abelian pseudo-tensor categories $\cA$ that are not of abelian type, see \S\ref{Klein}.
\end{enumerate}

\end{remark}

For the applications to symmetric tensor categories mentioned in the introduction,
the following observation is important. The same conclusion holds in the braided setting.
\begin{lemma}\label{LemSym}
Let $\cA$ be a pseudo-tensor category over $\bk$ with commutor. A tensor ideal in $\cA$ is of abelian type if and only if it is the kernel of a commutor-compatible tensor functor to a tensor category over $\bk$ with commutor. The statement remains valid if we replace `commutor' with `commutor with properties', such as a (symmetric) braiding.
\end{lemma}
\begin{proof}
For ease of notation, and without loss of generality, we focus on the zero ideal. If~$\cA$ admits a faithful tensor functor to a tensor category $\cC$, we can replace the tensor functor with the corresponding local abelian envelope $\cA\to\cC_\alpha$, see \ref{def:tensorcat}. By construction of $\cC_\alpha$ in \cite[\S 5.2]{HomAb}, the monoidal category $\Ind\cC_\alpha$ is a localisation with respect to a Grothendieck topology of the monoidal presheaf category (via Day convolution) of the kernel category (see \cite[Definition~3.1.2]{HomAb}) of~$\cA$. If $\cA$ has a commutor, then so has its kernel category, the presheaf category and its localisation, with inheritance of properties.
\end{proof}


\begin{remark}We comment on some variations of Definition~\ref{DefAbT}.
\begin{enumerate}
\item
One could weaken the definition of tensor ideals of abelian type to include kernels of tensor functors to tensor categories over field extensions $K:\bk$. Even when $\bk=\mC$ this is a genuinely more general notion, as follows for instance from combining \cite[Theorem~10.10]{OS} and Remark~\ref{rem:super}. Such ideals are still prime, but no longer always faithfully prime. For our purposes it seems preferable to work with Definition~\ref{DefAbT}.
\item Every tensor ideal is `of triangulated type', since we can embed $\cA$ faithfully into the tensor-triangulated homotopy category $K^b(\cA)$.
\item We can say that a tensor ideal is `of semisimple type' if it is the kernel of a tensor functor to a semisimple tensor category. This notion does not satisfy the analogue of Lemma~\ref{LemSym}. Indeed, the symmetric tensor category $\Ver_4^+$, see e.g. \cite{BEO}, admits an exact tensor functor to $\Vecc$, but no exact symmetric tensor functor to a symmetric semisimple tensor category.
\item One can weaken Definition~\ref{DefAbT} to include kernels of tensor functors to multi-tensor categories, see \cite[\S 4.1]{EGNO}. This gives a more general notion, since the intersection of two tensor ideals of abelian type is always the kernel of a tensor functor to a multi-tensor category, but not prime (so not of abelian type) unless one ideal is included in the other. A more subtle question is whether every {\em prime} ideal that is the kernel of a tensor functor to a multi-tensor category is of abelian type. If we restrict to symmetric functors and categories, the answer is easily seen to be affirmative, since symmetric multi-tensor categories are simply products of symmetric tensor categories, as follows from \cite[\S 4.3]{EGNO}. Using arguments as in Lemma~\ref{LemSym}, we can even extend this to pseudo-tensor categories with commutor in relation to arbitrary multi-tensor categories.
\end{enumerate}
\end{remark}

\subsection{A method for classification}\label{Sec:Method}
In order to relate tensor ideals of abelian type with classes of tensor ideals that are more intrinsically characterised, and the coarser notion of thick tensor ideals, we suggest to exploit the following diagram, as a refinement of \eqref{ObAll}:
\begin{equation}
\label{themap}\xymatrix{
\{\mbox{tensor ideals of abelian type}\}\ar[rrr]^-{\Upsilon=\Ob}\ar@{^{(}->}[d]^{\incl_1}&&&\{\mbox{prime thick tensor ideals}\}\\
\{\mbox{faithfully prime tensor ideals}\}\ar[urrr]^{\Ob}\ar@{^{(}->}[rrr]^{\incl_2}&&& \{\mbox{prime tensor ideals}\}\ar@{->}[u]_{\Ob}
}\end{equation}
The injections are inclusions of subsets, see Proposition~\ref{propnotab1}. Under fairly broad assumptions, the two upwards arrows are surjective, see Proposition~\ref{PropImax} below.

As we will see, none of the five maps above needs to be a bijection in general. However, in many cases they seem to be close enough to bijections that they are useful in classification efforts. It seems to be a principle that the maps in \eqref{themap} are bijections when sufficient `finiteness' conditions are satisfied, see for example Theorem~\ref{LemStr}(2) and Theorem~\ref{Thm:FinR} below.
Up to formulation, $\incl_1$ is shown to be a bijection in a large class of examples by O'Sullivan in \cite{OS}:
\begin{prop}\label{PropOS}
If $\bk$ has characteristic 0 and $\cA$ is symmetric and such that for every $X\in \cA$, there is $n\in\mN$ for which $\bk S_n\to \End(X^{\otimes n})$ is not injective, then every faithfully prime tensor ideal in $\cA$ is of abelian type.
\end{prop}
\begin{proof}
If a tensor ideal $\cI$ in $\cA$ is prime, then (by definition) $\cA/\cI$ is integral in the sense of \cite[\S 3]{OS}. Moreover, if $\cI$ is faithfully prime then the `fractional closure' of $\cA/\cI$ (into which $\cA/\cI$ faithfully embeds) is a symmetric pseudo-tensor category again over $\bk$, see \cite[Lemma~3.3]{OS}. The conclusion thus follows from \cite[Corollary~10.4]{OS}.
\end{proof}

We conclude this subsection by pointing out that in general the map $\Upsilon$ in \eqref{themap} need not be injective or surjective, even when working over $\bk=\mC$ and focusing on symmetric pseudo-tensor categories.

\begin{example}\label{ExNotInj}
If a symmetric tensor category has no objects of categorical dimension zero, it has no non-zero proper thick tensor ideals, but it can have several tensor ideals of abelian type. The typical example is the category of finite dimensional rational representations $\Rep_{\mC}\mG_a$, for $\mG_a$ the additive group scheme over $\mC$. Then $0$ and $\cN$ are two tensor ideals of abelian type that are both sent to $\bo$,
 showing that $\Upsilon$ is {\bf not injective} for $\cA=\Rep_{\mC}\mG_a$. In this case the non-zero tensor ideals form one countable chain by Theorem~\ref{ThmRigBij} and the fact a $\mG_a$-representation is equivalent to a vector space with a nilpotent endomorphism, so that indecomposable modules are classified by their dimension. Moreover, $\incl_1$ and $\incl_2$ are bijections (but the maps labelled $\Ob$ are only surjective).
\end{example}

Another natural example of the failure of $\Upsilon$ to be injective is given in Subsection~\ref{Klein}. That will also lead to an example where $\incl_2$ is not a bijection.

\begin{example}\label{MEG}
In \cite[\S 5.8]{De}, Deligne constructed an example of a symmetric pseudo-tensor with an endomorphism of trace 1 that squares to zero. It follows that it does not have any tensor ideals of abelian type, by Proposition~\ref{propnotab2}. In fact, the only proper tensor ideal is the zero ideal, as follows from Theorem~\ref{ThmRigBij} and the observation that composition
$$\Hom(\unit, X^{\otimes 2n})\otimes\Hom(X^{\otimes 2n},\unit)\to\End(\unit)=\bk$$ 
is a perfect pairing for the generating object $X$. In particular, it then follows that $\bo$ is prime (if $U\otimes V=0$, then the collection of objects $W$ with $W\otimes V=0$ forms a non-zero proper thick tensor ideal). 
Hence $\Upsilon$ is {\bf not surjective} for the example in \cite[\S 5.8]{De}. However, $\incl_2$ and the other occurrences of $\Ob$ in \eqref{themap} are all bijections.
\end{example}
Another example of the behaviour in Example~\ref{MEG} is given by $(\Rep GL)_\delta$ in characteristic $p>0$ with $\delta\not\in\mF_p$, see \S\ref{sec:RepGLt}. 

%

\subsection{Semiprimitive tensor ideals}\label{sec-fibre}

In this section we develop a sufficient condition, in Theorem~\ref{Thm:FinR}, under which the right triangle in \eqref{themap} consists of bijections. The intermediate results are applicable in much greater generality.

\subsubsection{}\label{minmax} Let $\cI,\cJ$ be tensor ideals in $\cA$ and $\bI$ a thick tensor ideal.
We have 
$$\Ob(\cI\cap\cJ)\;=\;\Ob(\cI)\cap\Ob(\cJ),$$
so it follows that there is a unique minimal element in $\Ob^{-1}(\bI)$, that we denote by $\bI^{\min}$. Explicitly, this is the tensor ideal of all morphisms that factor through an object in $\bI$.

It follows from locality of endomorphism algebras of indecomposable objects that
\begin{equation}\label{eq:sumideals}\Ob(\cI+\cJ)\;=\;\Ob(\cI)\vee \Ob(\cJ),\end{equation}
where the set of indecomposable objects in the join $\vee$ of two thick tensor ideals is the union of the sets of indecomposable objects in both thick tensor ideals.
Hence there is always a maximal element in $\Ob^{-1}(\bI)$ that we denote by $\bI^{\max}$. We will call the tensor ideals of the form $\bI^{\max}$ the {\bf semiprimitive} tensor ideals. This is motivated by the fact a tensor ideal $\cI$ is semiprimitive if and only if the `monoidal Jacobson radical', see Remark~\ref{rem:dic}, of $\cA/\cI$ is zero.

For an ideal $\cI$ in a $\bk$-linear category $\cB$, we denote by $\cI^n$, for $n\in\mN$, the corresponding power of the ideal, consisting of morphisms $\sum_i f_{i}^{(1)}\circ\cdots \circ f_{i}^{(n)}$ for $f_{i}^{(j)}\in\cI$. 

Recall the (non-monoidal) radical $\cR=\cR(\cA)$ from \ref{def:Jac}.

\begin{theorem}\label{Thm:FinR}
 Let $\cA$ be a pseudo-tensor category with commutor and assume $\cR(\cA)^M=0$ for some $M\in\mN$. Then $\Ob$ yields a bijection between prime tensor ideals and prime thick tensor ideals, with inverse $\bI\mapsto \bI^{\max}$. Hence a tensor ideal $\cI$ is prime if and only if it is semiprimitive and $\Ob(\cI)$ is prime. If $\bk$ is algebraically closed, every prime tensor ideal is faithfully prime.
\end{theorem}

In the rest of the subsection we prove Theorem~\ref{Thm:FinR}. In \cite[Proposition~2.1.7]{Sinf} the following proposition (in a different formulation) was proved for the oriented Brauer category. We adapt the ideas to generalise the result:

\begin{prop}\label{PropImax}
\label{LemImax}Assume that $\cA$ has a commutor.
For a prime thick tensor ideal $\bI$ in $\cA$, the ideal $\bI^{\max}$ is prime, and faithfully prime if $\bk$ is algebraically closed.
\end{prop}
\begin{proof}
If $f,g\not\in\bI^{\max}$, then the tensor ideals generated respectively by $f$ and $g$ contain respective objects $X$ and $Y$ not in $\bI$. If $f\otimes g\in\bI^{\max}$, it follows that $X\otimes Y$ must be in $\bI$, contradicting primeness of $\bI$. Hence $\bI^{\max}$ is prime.

Now assume that $\bk$ is algebraically closed. We will write the tensor product $X\otimes Y$ in $\cA$ simply as $XY$, for brevity and to avoid confusion with the tensor product over $\bk$. 

By replacing $\cA$ with $\cA/\bI^{\min}$, we can reduce to the case $\bI=\bo$. For each $X,Y\in\cA$ we define a set of morphisms $\cJ(X,Y)$ by the following exact sequence
$$0\to\cJ(X,Y)\to \cA(X,Y)\to\prod_{Z}\cA(X Z,Y Z)/\cR(X Z,Y Z),$$
where $Z$ runs over all objects in $\cA$ and the right arrow represents the tensor product $-Z$ followed by projection onto the quotient space. Since $\cR$ is an ideal, it follows that so is $\cJ$, and by construction it is a right (and thus two-sided, using the commutor) tensor ideal. As such it is the maximal tensor ideal contained in~$\cR$ and thus $\cJ=\bo^{\max}$. We thus need to prove that $\cJ$ is faithfully prime.

Take therefore arbitrary $X,X',Y,Y'\in \cA$ and consider the exact sequence
$$0\to\cJ(X,X')\otimes_{\bk}\cA(Y,Y')+\cA(X,X')\otimes_{\bk}\cJ(Y,Y')\to \cA(X,X')\otimes_{\bk}\cA(Y,Y')$$
$$\to\prod_{U,V}\left(\cA(X U,X' U)/\cR(X U,X'U)\right)\otimes_{\bk}\left(\cA(YV,Y'V)/\cR(YV,Y'V)\right).$$
For $x\in \cA(X,X')\otimes_{\bk}\cA(Y,Y')$ such that its image in $\cA(X Y,X' Y')$ is in $\cJ$, we need to show that it is sent to zero under the last arrow in the previous sequence for every pair $U,V$. However, since $\cJ$ is a tensor ideal, it is actually sufficient to deal with the case $U=V=\unit$.

The statement about $x$ needed to be proved by the last paragraph can, by linearity, be reduced to case where $X,X',Y,Y'$ are indecomposable, and thus without loss of generality $X'=X$ and $Y'=Y$. We have thus arrived special case of the following claim, for finite dimensional associative $\bk$-algebras $A,B,C$ (all non-zero, where $C=\cA(XY,XY)$ is non-zero by assumed primeness of $\bo$), with $A,B$ local, and an algebra morphism
\[A\otimes_{\bk}B\to C,\]
if $x\in A\otimes B$ is sent to an element of the Jacobson radical of $C$, it must be in $J\otimes B+A\otimes J$, where we denote by $J$ the Jacobson radical of both $A$ and $B$. This is indeed true since the Jacobson radical (for finite dimensional algebras) is nilpotent and, since $\bk$ is algebraically closed, the quotient by it is $\bk$. In particular $A\otimes B$ is local with nilpotent maximal ideal $J\otimes B+A\otimes J$. Hence, if $x$ is not in 
$J\otimes B+A\otimes J$, it is a unit in $A\otimes B$ and not sent to a nilpotent.
\end{proof}



\subsubsection{} 

For an ideal $\cI$ in a $\bk$-linear category $\cB$, we write $\cI^\infty=\cap_n \cI^n$. In a pseudo-tensor category $\cA$, for any $X,Y\in\cA$, we have $\cI^\infty(X,Y)=\cI^n(X,Y)$ for some $n$, by finite-dimensionality.

For a tensor ideal $\cI$ in $\cA$, let $\Rad(\cI)$ be
 the maximal ideal $\cJ$ for which $\cJ^\infty\subset\cI$. We explain why this is well-defined. Firstly, there exists some maximal ideal $\cJ$ for which $\cJ^\infty\subset\cI$ by Zorn's lemma. Indeed, for a chain $\cJ_i\subset \cJ_{i+1}$ of such ideals, we can use finite-dimensionality of morphism spaces to conclude that for any given $X,Y\in\cA$, we have $(\cup_i \cJ_i)^\infty(X,Y)=(\cJ_m)^\infty(X,Y)$ for high enough $m$, showing that $(\cup_i\cJ_i)^\infty\subset\cI$. Furthermore,  $(\cJ_1+\cJ_2)^\infty=\cJ_1^\infty+\cJ_2^\infty$, by the previous paragraph, showing uniqueness of the maximal $\cJ$ for which $\cJ^\infty\subset\cI$. Finally, we observe that $\Rad(\cI)$ is again a tensor ideal. Similarly, let $\rad(\cI)$ be
 the maximal ideal $\cJ$ for which $\cJ^n\subset\cI$ for some $n\in\mN$. Note that we have
 $$\rad(\cI)\;\subset\;\Rad(\cI)\;\subset\; \left(\Ob(\cI)\right)^{\max}.$$
In Example~\ref{ExNotInj}, we have
 $0=\rad(0)$ and $\Rad(0)=\cN.$
%
 
 \begin{prop}\label{prop:minmax}
 Let $\cA$ be a pseudo-tensor category and assume that $\cR(\cA)^\infty=0$. Consider a tensor ideal $\cI$ and set $\bI=\Ob(\cI)$.
 Then $\bI^{\max}=\Rad(\cI)$ and $\bI^{\min}=\cI^\infty$.
 \end{prop}
\begin{proof}
First we prove the equality $\bI^{\min}=\cI^\infty$. Since the condition $\cR(\cA)^\infty=0$ is inherited by quotient categories of $\cA$, we can replace $\cA$ with a monoidal quotient category to reduce to the case $\bI=0$. Now we have $\cI\subset\cR(\cA)$ and thus indeed $\cI^\infty=0$.

Now, we can apply the equality $\bI^{\min}=\cI^\infty$ to the case where $\cI$ is $\bI^{\max}$, which shows $(\bI^{\max})^\infty=\bI^{\min}$, so in particular $(\bI^{\max})^\infty\subset\cI$, or in other words $\bI^{\max}\subset \Rad(\cI)$. That this inclusion is also an equality follows from $\Ob(\cJ^\infty)=\Ob(\cJ)$ for any tensor ideal.
\end{proof}

\begin{lemma}\label{Lem:rad}
If $\cI$ is a prime tensor ideal in $\cA$, then $\Ob(\cI)$ is prime and $\rad(\cI)=\cI$.
\end{lemma}
\begin{proof}
Since $\rad(\cI)^n\subset\cI$ for some $n\in\mZ_{\ge 1}$, for any $f:X\to Y$ in $ \rad(\cI)$ we have $f^{\otimes n}\in \cI$, by writing $f^{\otimes n}$ as a composition of $n$ morphisms $X^{\otimes i}\otimes f\otimes Y^{\otimes n-i-1}\in \rad(\cI)$. Hence $f\in \cI$.
\end{proof}

\begin{proof}[Proof of Theorem~\ref{Thm:FinR}]
We start by proving an auxiliary result: for any ideal $\cJ$ in $\cA$, we have $\cJ^M=\cJ^\infty$, or equivalently, $\cJ^M\subset\cJ^{M+1}$. Indeed, an arbitrary element of $\cJ^{M}$ can be written as
$$\sum_i f_{i,1}\circ f_{i,2}\circ\,\cdots\, \circ f_{i,M}$$
with each $f_{i,j}$ a morphism in $\cJ$ with indecomposable source and target. Ignoring zero terms, it thus follows that for each $i$ we must have an $f:=f_{i,j}$ that is an isomorphism, so that we can replace $f$ with $f(f^{-1}f)\in \cJ^2$ to show that each non-zero term is thus actually in $\cJ^{M+1}$.

 By Proposition~\ref{LemImax}, $\bI^{\max}$ is (faithfully) prime for $\bI$ prime. On the other hand, if $\cI$ is a prime tensor ideal in $\cA$, then by Lemma~\ref{Lem:rad} the thick tensor ideal $\bI:=\Ob(\cI)$ is prime and $\rad(\cI)=\cI$. Let us show that $\cI=\bI^{\max}$. By Proposition~\ref{prop:minmax}, it suffices to prove that $\Rad(\cJ)=\rad(\cJ)$ for ideals $\cJ$ in $\cA$. But that follows immediately from the previous paragraph.
\end{proof}

\subsection{Existence of abelian envelopes}
In this section we develop a result to deal with the surjectivity question of $\incl_1$ in \eqref{themap}. The theorem is actually a reformulation of a result of Stroi\'nski, see \cite[Proposition~4.26]{Stroinski}, as part of a more general theory. For the reader's convenience we give a direct proof of the precise result we need, similar in philosophy to Stroi\'nski's approach in \cite[\S 4]{Stroinski}. We say that an element $x$ in a partially ordered set $(\Lambda,\le)$ has a `unique cover' $y$ if $x<y$ and all $z\in\Lambda$ with $x< z$ satisfy $y\le z$.

\begin{theorem}\label{ThmAbEnv}\label{LemStr} Let $\cA$ be a pseudo-tensor category over $\bk$ with commutor.

\begin{enumerate}
\item Assume that $\cA$ has a unique minimal non-zero tensor ideal $\cQ$ and that $\Ob(\cQ)$ is non-zero, but only has finitely many indecomposable objects. Then $\cA$ has an abelian envelope.
\item Let $\bI$ be a thick tensor ideal in $\cA$ and assume that there is a unique cover $\bI<\bJ$ in the inclusion order on thick tensor ideals in $\cA$.
If the number of indecomposable objects in $\bJ\,\backslash\, \bI$ is finite, then $\bI^{\max}$ is of abelian type and $\cA/\bI^{\max}$ admits an abelian envelope.
\end{enumerate}
\end{theorem}
\begin{proof}
Starting with (1), write $\bQ=\Ob(\cQ)$, and observe that $\bQ$ is the unique minimal non-zero thick tensor ideal in $\cA$, since $\Ob$ is surjective. We also use the symbol $\bQ$ for the full subcategory of $\cA$ on the corresponding objects. This is thus a semigroup subcategory (see \cite[Remark~2.2.9]{EGNO}) and we label its indecomposable objects as $Q_i$ for $1\le i\le n$.  We have multiplication rules
\begin{equation}\label{QtQ}
Q_j\otimes Q_i\;\simeq\; \bigoplus_{l} Q_l^{c^l_{ji}},\qquad c_{ji}^l\in\mN,
\end{equation}
from which we can assign Frobenius-Perron dimensions $q_i:=\FPdim Q_i\in\mR_{>0}$, extending by linearity to a ring homomorphism $K_0^{\oplus}(\bQ)\to \mR,$ from the split Grothendieck ring,
see \cite[\S 3.3]{EGNO} (note that $K_0^{\oplus}(\bQ)$ is transitive by rigidity).
We will prove that $\bQ$ is a splitting ideal in $\cA$, in the sense of \cite[Definition~2.2]{BEO}.

Consider an arbitrary morphism $f:Q\to Q'$ in $\bQ$ and set $q =\FPdim (Q)$. For each $i$, we consider the morphism
$$Q_i\otimes f:\; Q_i\otimes Q\to Q_i\otimes Q'.$$
We let $u_i$ denote the minimal Frobenius-Perron dimension of a direct summand of $Q_i\otimes Q$ for which the corresponding idempotent $e$ is such that $(Q_i\otimes f)\circ (1-e)$ is split. This exists, as there are only finitely many values of Frobenius-Perron dimensions of direct summands for any given object. For example, $u_i=0$ if and only if $Q_i\otimes f$ is split, and $u_i=q_iq$ implies that $Q_i\otimes f$ is in $\cR(\bQ)$  (although the converse is not true, for example if $f=0$ then $u_i=0$).

We can similarly define a minimal dimension $u_{ij}$ based on the morphism $Q_j\otimes Q_i\otimes f$. Firstly by using that applying $Q_j\otimes-$ to a split morphism remains split and secondly by applying \eqref{QtQ}, we find 
\begin{equation}
\label{ineq}
q_ju_i\;\ge\;u_{ij}\;=\;\sum_l c_{ji}^l u_l.
\end{equation}
Also, by \cite[3.3.6(2)]{EGNO} there exist $d_i>0$ (which a posteriori will be the Frobenius-Perron dimension of the simple top of $Q_i$ in the abelian envelope) for which
$$\sum_i d_ic_{ji}^l\;=\; q_jd_l.$$
Combining this with \eqref{ineq} shows
$$\sum_i d_i q_j u_i\ge \sum_id_i u_{ij}\;=\;\sum_{i,l}d_i c_{ji}^l u_l\;=\;\sum_ld_lq_j u_l,$$
hence the inequality is an equality which, by \eqref{ineq} and $d_i>0$, can only hold if $u_{ij}=q_ju_i$ for every $i,j$. 

Now if $u_i=0$ for all $i$ then $Q_i\otimes f$ is split. Conversely, assume $u_i>0$, so that we can define a non-zero morphism $g$ by restricting $Q_i\otimes f$ to a (non-zero) minimal summand of $Q_i\otimes Q$ (with dimension $u_i$) as defined in the previous paragraph. Since \eqref{ineq} is actually an equality, the discussion at the end of the second paragraph of the proof shows that $Q_j\otimes g$ is in the radical $\cR(\bQ)$, for each $j$.
We claim that then actually $X\otimes g$ is in the radical of $\bQ$ for all $X\in \cA$. Indeed, we know that $P\otimes P^\ast\otimes X\otimes g$ is in $\cR(\bQ)$ for all $P\in \bQ$, so that $\id_P$ cannot be in the (non-tensor) ideal generated by $X\otimes g$. Hence the tensor ideal in $\cA$ generated by $g$ contains none of the $Q_i$, contradicting the minimality of~$\bQ$. We thus conclude that in fact $u_i=0$ for all $i$ (implying that $\bQ$ is a splitting ideal).

Hence, following \cite[\S 2.9 and \S2.10]{BEO}, we have a tensor category $\cC(\cA,\bQ)$ with tensor functor $F:\cA\to \cC(\cA,\bQ)$ such that the kernel consists of those morphisms $h$ for which $P\otimes h=0$ in $\cA$ for all $P\in\bQ$. Let $\cJ$ denote this ideal. By minimality of $\cQ$, we must have $\cQ\subset\cJ$ or $\cJ=0$. The former is not possible, since $P\otimes P^\ast\not=0$ for non-zero $P\in\bQ$. Thus $\cJ=0$, meaning that $F$ is faithful and it follows from \cite[Theorem~2.42]{BEO} or Lemma~\ref{LemGF} that $\cC(\cA,\bQ)$ is the abelian envelope.

For part (2), we can observe that $\bI^{\max}$ has a unique cover in the inclusion order on tensor ideals in $\cA$. Indeed, for every strict inclusion $\bI^{\max}\subset\cJ$ of tensor ideals, we have $\bJ\subset\Ob(\cJ)$, so the proper tensor ideal given by the intersection of all such $\cJ$ does not equal $\bI^{\max}$.
It then suffices to apply part (1) to $\cA\,/\,\bI^{\max}$.
\end{proof}

\begin{remark}
Theorem~\ref{ThmAbEnv}(1) extends, with same proof, to categories without commutor, but then the two-sided tensor ideal $\cQ$ needs to be minimal as a left tensor ideal.
\end{remark}

Theorem~\ref{LemStr}(2) shows in particular that $\bI$ is prime. We show that this is true even without the finiteness assumption, and provide a partial converse implication.
\begin{lemma}\label{LemPrimeCover}
Assume that $\cA$ has a commutor and let $\bI$ be a thick tensor ideal in $\cA$.
\begin{enumerate}
\item If there is a unique cover $\bI<\bJ$ in the inclusion order, then $\bI$ is prime.
\item If
$\bI$ is prime and there are only finitely many thick tensor ideals in $\cA$, then there is a unique cover $\bI<\bJ$ in the inclusion order.
\end{enumerate} 
\end{lemma}
\begin{proof}
For part (1), let $Q_1,Q_2$ be two objects not in $\bI$. Then $\bJ$ is included in the thick tensor ideal generated by $\bI$ and $Q_1$. From this we can conclude that any indecomposable $X\in\bJ\backslash \bI$ is a direct summand of $Q_1\otimes A$ for some $A\in\cA$. Similarly $X$ is a summand of $Q_2\otimes B$. Hence $X$, as a direct summand of $X\otimes X^\ast\otimes X$, is a direct summand of $Q_1\otimes Q_2\otimes C$ (for $C=A\otimes B\otimes X^\ast$). This implies that $Q_1\otimes Q_2$ is not in $\bI$.

For part (2), if there would be two different coverings $\bJ_1$ and $\bJ_2$ of $\bI$, then $\bI=\bJ_1\cap \bJ_2$, which contradicts primeness of $\bI$.
\end{proof}


\section{Finite groups}\label{sec:fingrp}
Let $\bk$ be a field of characteristic $p>0$.

\subsection{Abelian envelopes}

In this section we fix a finite group $G$, with normal subgroup $N$.

\subsubsection{}\label{subsecT}
We define the tensor ideal $\cI_N$ in $\Rep_{\bk}G$ as the preimage under the restriction functor of the maximal ideal $\cN$ in $\Rep_{\bk}N$. By construction $\cI_N$ is of abelian type, see Example~\ref{preimN}. 

There is a canonical action of $G$ on $\Rep_{\bk}N$, so that the action of $N$ is isomorphic to the trivial action, producing an action of $Q$. To write it explicitly, we choose coset representatives $\{g_i\}$ for $Q :=G/N$ so that we can define an action of $Q $ on $\Rep_{\bk} N$. This action sends $g_i N$ to the tensor auto-equivalence $T_{g_i}$ of $\Rep_{\bk }N$ that sends a representation $V$ to the same vector space, but with new action of $n\in N$ given by the previous action of $g_i^{-1}ng_i$. Moreover, for coset representatives $g_i,g_j,g_l$ such that $g_iNg_jN=g_lN$, we have $n_0\in N$ with $g_ig_j=g_ln_0$ and then
the isomorphism $T_{g_i}\circ T_{g_j}\Rightarrow T_{g_l}$ is, when evaluated on a representation $V$, given by the morphism of vector spaces $v\mapsto n_0 v$.
This leads to the equivalence of tensor categories on the top row of the diagram
$$\xymatrix{\Rep_{\bk}G\ar[rr]^-{\sim}\ar[d]&& (\Rep_{\bk}N)^Q\ar[d]\\
 (\Rep_{\bk }G)/\cI_N\ar[rr]^-\Phi&&\left(\overline{\Rep_{\bk}N}\right)^Q}$$
see \cite[4.5.13]{EGNO}. Since $\overline{\Rep_{\bk}N}$ is just the quotient of $\Rep_{\bk}N$ by $\cN$, it is equipped with the same action of $Q $ and we obtain the above commutative diagram of symmetric tensor functors, where the horizontal arrows are faithful.

\begin{theorem}\label{PropQN} 
The symmetric tensor functor
$$\Phi:(\Rep_{\bk }G)/\cI_N\;\to\;\left(\overline{\Rep_{\bk}N}\right)^Q$$
is an abelian envelope. 
\end{theorem}
\begin{proof}
We denote the functor forgetting the $Q $-action on $(\overline{\Rep_{\bk}N})^Q$ by $\Res$ and the composite in the square before the theorem by $\widetilde{\Phi}$. We obtain a commutative square
$$\xymatrix{
\Rep_{\bk}G\ar@/^/[rr]^{\Res^G_N}\ar[d]_-{\widetilde{\Phi}}&&\Rep_{\bk}N\ar[d]^-{V\mapsto \overline{V}}\ar@/^/[ll]^{\Ind^G_N}\\
\left(\overline{\Rep_{\bk}N}\right)^Q\ar@/^/[rr]^{\Res}&& \overline{\Rep_{\bk}N}.\ar@/^/[ll]^{\Ind}
}$$
It is a classical fact, see \cite[Lemma~4.6]{DGNO}, that $\Res$ is bi-adjoint to a functor $\Ind$ which also leads to a commutative square with $\Ind^G_N$.

We thus have isomorphisms
\begin{equation}\label{2iso}
\Hom(\widetilde{\Phi}(\Ind^G_N(V)), R)\simeq\Hom(\Ind(\overline{V}),R )\simeq \Hom(\overline{V}, \Res R),
\end{equation} 
natural in $V\in \Rep_{\bk}N$ and $R\in \left(\overline{\Rep_{\bk}N}\right)^Q$.
It follows from semisimplicity of  $\overline{\Rep_{\bk}N}$ that the objects $\widetilde{\Phi}(\Ind^G_N(V))$ with $V\in \Rep_{\bk}N$ are precisely the projective objects in $\left(\overline{\Rep_{\bk}N}\right)^Q$ and that every object in $\left(\overline{\Rep_{\bk}N}\right)^Q$ is a quotient of such a projective object.

Next we claim that, for $V\in \Rep_{\bk}N$ and $W\in \Rep_{\bk}G$, the morphism
\begin{equation}
\label{eqsurj}\Hom_G(\Ind^G_NV,W)\;\xrightarrow{\widetilde{\Phi}}\; \Hom(\widetilde\Phi(\Ind^G_N(V)),\widetilde\Phi(W))
\end{equation}
is surjective. Indeed, via adjunction and \eqref{2iso}, we can rewrite this morphism as the quotient
$$\Hom_N(V,\Res^G_NW)\;\tto\; \Hom(\overline{V}, \overline{\Res^G_NW}).$$

We can now apply Lemma~\ref{LemGF}. We have verified already condition (G) and condition (F) follows quickly from the surjectivity of \eqref{eqsurj}.
\end{proof}

\begin{remark}
In general the functor $\Phi$ is not full, see Example~\ref{RemEnvC} below.
\end{remark}

\begin{example}
If $N$ contains a Sylow $p$-subgroup, then $\cI_{\cN}=\cN$ and $(\Rep_{\bk}G)/\cN$ is simply $(\overline{\Rep_{\bk }N})^Q$ by \cite[Corollary~3.12]{EO}.
\end{example}

\subsection{Cyclic groups}

In this section, we consider the tensor category $\Rep_{\bk} C_{p^n}$ for the cyclic group $C_{p^n}$ with $n\in\mZ_{>0}$. We will classify all tensor ideals, thick tensor ideals and their basic properties, leading to the following main result:

\begin{theorem}\label{ThmCpn}
For $\Rep_{\bk} C_{p^n}$, all maps in \eqref{themap} are bijections. Moreover, every quotient with respect to a tensor ideal of abelian type admits an abelian envelope.
\end{theorem}

The proof will occupy the rest of this subsection. We could also apply Theorem~\ref{Thm:FinR}, but we give a direct proof here.

\subsubsection{} Set $\cC_n:=\Rep_{\bk} C_{p^n}$.  If we pick a generator $\sigma\in C_{p^n}$, then $\bk C_{p^n}\simeq \bk [x]/x^{p^n}$ with $x:=\sigma-1$. The indecomposable modules in $\cC_n$ are labelled as $J_i\simeq\;\bk[x]/x^i$, for $1\le i\le p^n$. 

By Theorem~\ref{ThmRigBij}, the proper tensor ideals in $\cC_n$ are given by
$$0=\cI_{p^n}\subset\cI_{p^n-1}\subset \cdots\subset\cI_2\subset\cI_1\subset\cC_n,$$
where the ideal $\cI_j$, for $j<p^n$, is generated by a monomorphism
$\unit=J_1\hookrightarrow J_{j+1}.$

\subsubsection{} We have unique subgroups $C_{p^l}<C_{p^n}$, for all $l\le n$, yielding short exact sequences
$$1\to C_{p^l}\to C_{p^n}\to C_{p^{n-l}}\to 1.$$
In particular, for $0\le m\le n$, we have tensor subcategories $\cC_m\subset\cC_n$, which satisfy $J_i\mapsto J_i$ for $i\le p^m$, and restriction functors $\cC_n\to \cC_m$. It follows that the intersection of $\cC_m\subset \cC_n$ with the ideal $\cI_i$ in $\cC_n$ is $\cI_i$ for $i\le p^m$ and $0$ for $i\ge p^m$.

\begin{lemma}\label{LemVic}
The ideals $\cI_{p^m}$, for $0\le m\le n$ are of abelian type. More concretely, $\cI_{p^m}$ is the kernel of the composite
$$\cC_n\to\cC_{n-m}\to\overline{\cC_{n-m}}.$$
\end{lemma}
\begin{proof}
The functor $\cC_n\to\cC_{n-m}$ sends $J_{p^m}$ to $\unit^{p^m}$ and $J_{p^m+1}$ to $J_2\oplus\unit^{p^m-1}$ and moreover, the inclusion $\unit\hookrightarrow J_{p^m+1} $ to the sum of $\unit\hookrightarrow J_2$ and the zero morphism $\unit\to\unit^{p^m-1}$. Hence $\unit\hookrightarrow J_{p^m}$ is not sent to zero by the composite functor, but $\unit\hookrightarrow J_{p^m+1}$ is.
\end{proof}

\begin{lemma}\label{lem:cyclic}
The proper thick tensor ideals in $\Rep_{\bk} C_{p^n}$ form a chain
$$\bo=\bI_n\;\subset \bI_{n-1}\;\subset\;\cdots\;\subset\; \bI_1\;\subset\; \bI_0.$$
The indecomposable objects in $\bI_m$ are given by
$\{J_{a p^{m+1}} \mid 1\le a\le p^{n-m-1}\}.$
The map $\Ob$ in \eqref{ObAll} is given by
$$\cI_i\;\mapsto \;\bI_{\lfloor\log_p(i)\rfloor}\;=\; \bI_{m},\quad\mbox{with } \; p^m\le i<p^{m+1}.$$
In particular, $\bI_m^{\max}=\cI_{p^m}$.
\end{lemma}

\begin{proof}
First we observe that the modules $J_{a p^{m+1}}$ are precisely the indecomposable objects in $\Ob(\cI_{p^m})$ that become negligible after restriction $\cC_n\to\cC_{n-m}$.
Indeed, this condition means that each summand over $\bk[x^{p^m}]/(x^{p^n})$ needs to be of dimension divisible by $p$.

 It remains to verify that, for $m<n$ and $p^m\le i<p^{m+1}$, in the chain
$$\Ob(\cI_{p^{m+1}})\;\subset\; \Ob(\cI_i)\;\subset\;\Ob( \cI_{p^{m}})$$
the right inclusion is an equality. 
For this, we can observe that $J_{p^{m+1}}\in \Ob(\cI_i)$, since $\cI_i$ intersects $\cC_{m+1}\subset\cC_n$ non-trivially, and must thus contain the ideal of projective objects in $\cC_{m+1}$. We can then observe that the thick tensor ideal generated by $J_{p^{m+1}}$ contains all the modules $J_{ap^{m+1}}$. Indeed, the tensor product
$J_{p^{m+1}}\otimes J_{ap^{m+1}}$
contains $J_{ap^{m+1}}$ as a direct summand, and this follows from observing that the tensor product must be in $\bI_m$, and that under restriction $\cC_n\to\cC_{n-m-1}$, the module  $J_{ap^{m+1}}$ is sent to $(J_a)^{p^{m+1}}$.
\end{proof}

\begin{prop}\label{PropEnvCpn}
The pseudo-tensor category $\cC_n/\cI_{p^m}$ for $0\le m\le n$, admits the abelian envelope
$$\cC_m\boxtimes \overline{\cC_{n-m}}\;=\;\Rep C_{p^m}\boxtimes\overline{\Rep C_{p^{n-m}}}.$$
\end{prop}
\begin{proof} By Lemma~\ref{LemVic}, the ideal $\cI_{p^m}$ is the preimage of the ideal of neglibile morphisms in $\Rep C_{p^{n-m}}$.
By an application of Theorem~\ref{PropQN}, we find that the abelian envelope is 
$\overline{\Rep C_{p^{n-m}}}^{C_p^m}$. We only need to prove that the action of $C_{p^m}$ becomes trivial on $\overline{\Rep C_{p^{n-m}}}$. If $\sigma$ is the generator of $C_{p^n}$, we choose as coset representatives of $C_{p^n}/C_{p^{n-m}}$ the set 
$$\{g_i:=\sigma^i\mid 0\le i <p^m\}$$

Since $C_{p^n}$ is abelian, every functor $T_{g_i}$ from \ref{subsecT} is the identity functor. Also the natural transformations from \ref{subsecT} are the identity, except when $i+j\ge p^{m}$, in which case the natural transformation (of the identity functor) is given by the action of $\sigma^{p^{m}}$. Now, the difference of the identity transformation and the above transformation is nilpotent on each indecomposable object $\Rep C_{p^{n-m}}$ and hence vanishes in the semisimplification.
 \end{proof}

\begin{remark}\label{RemEnvC}
\begin{enumerate}
\item The abelian envelope is never full for $0<m<n$. Indeed, one can observe that $\unit\to J_{p^{m}+1}$ is sent to zero, while $J_{p^{m}+1}$ is sent to an object $X$ that has $J_{p^{m}-1}\boxtimes \unit$ as a direct summand, so that the space of morphisms $\unit\to X$ in the abelian envelope is not zero.
\item The category $\overline{\Rep C_{p^{n-m}}}$ has been studied in \cite[\S 5]{CEO-as}.
\end{enumerate}

\end{remark}

\begin{proof}[Proof of Theorem~\ref{ThmCpn}]
The only thing that remains to be proved is that $\cI_j$ is not prime when $j$ is not a power of $p$.
Consider $0< a,b< p^n$ with
$$\binom{a+b}{a}\;\not\equiv\;0\mod p.$$
Then the submodule generated by $1\otimes 1$ in $J_{a+1}\otimes J_{b+1}$ has socle spanned by $x^{a}\otimes x^{b}$. Hence the image of 
$$(\unit\hookrightarrow J_{a+1})\otimes (\unit\hookrightarrow J_{b+1})$$
in $J_{a+1}\otimes J_{b+1}$ is contained inside a summand $J_{i+1}$ with $i\ge a+b$ (this is actually an equality). In conclusion, for such $a,b$ we observed that the tensor product of two morphisms that are not in $\cI_{a+b}$ ends up in $\cI_{a+b}$, so that $\cI_{a+b}$ is not prime.

Now if $1\le j\le  p^n$ is not a power of $p$, then, by Lucas' Theorem, there exists $0<a< j$ with
$\binom{j}{a}$ not divisible by $p$,
so that $\cI_j$ is not prime.
\end{proof}

\subsubsection{Generalisation}\label{formal}
Let $F(u,v)\in \bk[[u,v]]$ be a one-dimensional formal group law over $\bk$. For every $n\in \mZ_{>0}$, there is a cocommutative Hopf algebra structure $H_n$ on $\bk[x]/(x^{p^n})$, with
$$\Delta(x)\;=\; F(x\otimes 1,1\otimes x).$$
For $m<n$, we have Hopf algebra inclusions $H_m\subset H_n$, $x\mapsto x^{p^{n-m}}$ and morphisms $H_n\to H_m$, $x\mapsto x$. For example, taking the multiplicative formal group law $u+v+uv$ recovers the example $H_n\simeq \bk C_{p^n}$ from this section. The additive group law gives $H_n=\bk[\alpha_{p^n}]$, the coordinate algebra of the $n$-th Frobenius kernel of the additive group $\mG_a$.

One can verify that all observations in the current section, with the exception of Proposition~\ref{PropEnvCpn}, remain valid with identical proof for the tensor categories $\cC_n^F:=H_n$-${\sf Mod}$. In particular all maps in \eqref{themap} are bijections.

\begin{problem}
Show that Proposition~\ref{PropEnvCpn} remains valid in the generality of \ref{formal}.
\end{problem}
The complication is that we cannot use Theorem~\ref{PropQN}, although the abelian envelope is still guaranteed to exist by Theorem~\ref{ThmAbEnv}. To identify the envelope one needs for instance to understand the endomorphism algebra of the functor $\cC^F_n\to \overline{\cC^F_{n-m}}$. Note that the semisimplification of $\cC_n^F$ is known to be independent of $F$, as follows from independence of the Green ring of $\cC_n^F$ on $F$ and \cite[\S 5]{CEO-as}.

\subsection{The Klein 4-group}\label{Klein}We let $\bk$ be an algebraically closed field of characteristic 2.
The tensor ideals of abelian type in $\Rep_{\bk} C_2^2$ have been classified in \cite[Theorem~6.1.2]{CF}. In the notation of \cite[Lemma~6.3.6]{CF}, they are indexed by 
$$\left(\mN\sqcup \{\infty\}\right)\times \mP^1(\bk)\;\sqcup\; \{\infty\}$$
as $\cJ_i^\lambda$ (with $i\in\mN\cup\{\infty\}$ and $\lambda\in\mP^1(\bk)$) and $\cJ_\infty=\cN$, except that $\cJ_1^\lambda$ only exists if $\lambda\in\mP^1(\mF_2)$. We actually slightly adapted notation for convenience below, as $\cJ_0^\lambda$ for us is denoted by $\cJ_1^\lambda$ in \cite{CF}, for $\lambda\not\in\mP(\mF_2)$. Moreover, by \cite[Theorem~6.1.2]{CF} they all admit abelian envelopes, which are
$\Rep C_2^2$, $\Rep C_2$, $\Rep\alpha_2 $, $\Ver_4$, $\Vecc$ or $\Vecc_{\mZ}$.

It turns out that $\Upsilon$ in \eqref{themap} is `almost' a bijection.
\begin{prop}
For $\Rep C_2^2$, the map $\Upsilon$ in \eqref{themap} is a surjection
such that all fibres are singletons, except that one fibre has the cardinality of $\bk$.
\end{prop}
\begin{proof}
It follows for instance from \cite[Theorem~6.1.2]{CF} that only the ideals $\cJ_\infty^\lambda$ and $J_\infty$ map to the same thick tensor ideal. Next, we prove surjectivity.

It follows from the tensor product rules, see for instance \cite[Lemma~6.2.1]{CF}, that proper thick tensor ideals can be parametrised by all functions
$$f: \mP^1(\bk)\;\to\; \mN\sqcup\{\infty\},$$
where only elements of $\mP^1(\mF_2)$ are allowed to be sent to $1$. For such a function, the indecomposable objects in the corresponding thick tensor ideal are all $A_m(\lambda)$ for $m\le f(\lambda)$. For such an ideal to be prime, we require $f(\lambda)=\infty$ for all but one $\lambda$. The corresponding thick tensor ideal is then precisely $\Ob(\cJ^\lambda_{f(\lambda)})$.
\end{proof}

Recall that $\bN$ always denotes the unique maximal thick tensor ideal (of negligible objects) and $\bN^{\min}$ is the minimal tensor ideal with $\Ob(\bN^{\min})=\bN$.

\begin{prop}
The ideal $\bN^{\min}$ is prime but not faithfully prime, so not of abelian type.
\end{prop}
\begin{proof}Set $\cA=\Rep C_2^2/\bN^{\min}$.
By the classification of modules recalled in \cite[\S 6.1]{CF}, the indecomposable objects in $\cA$ are given by the images of the representations $\Omega^j=\Omega^j\unit$, $j\in\mZ$, where $\Omega X$ stands for the kernel of the projective cover of $X$. In other words, $\Omega^{-n}$ for $n>0$ is $2n+1$-dimensional, with basis 
$$u_1,\ldots, u_n,v_0,\ldots,v_n$$
so that $(1+g_1)u_i=v_{i-1}$ and $(1+g_2)u_i=v_i$ for $g_1,g_2$ the generators of $C_2^2$ and $(1+g_j)v_i=0$ for $1\le j\le 2$ and $0\le i\le n$. The representation $\Omega^n$ is its dual. The tensor product in $\cA$ is given by $\Omega^n\otimes \Omega^m=\Omega^{n+m}$. Next we claim that 
$$\cA(\unit,\Omega^n)=\begin{cases}
0&\mbox{ if }\;n>0,\\
\Hom_{C_2^2}(\unit,\Omega^n)\simeq \bk^{1-n} &\mbox{ if }\;n\le 0.
\end{cases}$$
Indeed, for $n>0$, the socle of $\Omega^n\simeq (\Omega^{-n})^\ast$ is contained in the kernel of $v_0:(\Omega^{-n})^\ast\to\bk$, which is an even-dimensional indecomposable submodule of $\Omega^n$. In other words, every morphism $\unit\to \Omega^n$ factors through a negligible module. On the other hand, since every non-projective indecomposable negligible module $N$ in $\Rep C_2^2$ is a quotient of $\Omega^{-m}$, for some $m>0$, by a 1-dimensional submodule, it follows that the morphisms $N\to\Omega^n$, for $n<0$, take values in the socle, so that composites $\unit\to N\to \Omega^n$ are zero. The same conclusion holds if $N$ is projective.

That $\bN^{\min}$ is not faithfully prime now follows by the dimension count
$$\bk^4\simeq\Hom(\unit,\Omega^{-1})\otimes_{\bk}\Hom(\unit,\Omega^{-1})\;\xrightarrow{\otimes}\Hom(\unit,\Omega^{-2})\simeq\bk^3.$$

The above also shows that we have isomorphisms of graded vector spaces
\[\bigoplus_{m\ge 0}\cA(\unit,\Omega^{-m})\;\simeq\; \bigoplus_{m\ge 0}\Hom_{C_2^2}(\unit,\Omega^{-m})\;\simeq\;H^{\ast}(C_2^2,\bk)\;\simeq\; \bk[x,y],\]
where $H^\ast$ stands for the cohomology ring, where $x,y$ live in degree 1. Moreover, if we consider the binary operation $(f,g)\mapsto f\otimes g$ for homogeneous elements of the left-hand side then we obtain an isomorphism of graded $\bk$-algebras. In particular, the fact that $\bk[x,y]$ is an integral domain translates to the fact that $\bN^{\min}$ is prime. (In this interpretation, failure of faithful primeness is captured by $x\otimes y + y\otimes x\mapsto 0$ under multiplication on $\bk[x,y]$.)
\end{proof}


\begin{remark}
 We can also see that $\bN^{\min}$ is not of abelian type by Remark~\ref{RemNecCond} and the observation that the morphism space $\unit\to\Omega^{-1}$ is two-dimensional in $\Rep C_2^2/\bN^{\min}$.

\end{remark}

\begin{remark}\label{rem:super}
Let $\bk$ be an algebraically closed field with $\mathrm{char}(\bk)\not=2$. Then we can define a symmetric tensor category $\cC$ as the representation category of the Lie superalgebra with underlying supervector space $\bk^{0|2}$, which is an adaptation of the tensor category in the current section to other characteristics. Using the exact same arguments as above, the ideal $\bN^{\min}$ is prime but not faithfully prime. Alternatively, we can consider the zero ideal in the category of vector bundles on $\mP^1$.
\end{remark}


\section{Quantum groups: preparation}\label{sec:quantum1}
In this section we prepare the background for classifying prime tensor ideals for quantum tilting modules. We recall some basic facts and connect some dots in the literature.

\subsection{Conventions}\label{SecConv}
For this section and the next two, we set $\bk=\mC$.

\subsubsection{}\label{sec:condl}
Let $\fg$ be a finite-dimensional semisimple Lie algebra of rank $r$ with fixed triangular decomposition
$\fn^-\oplus\fh\oplus\fn^+,$ and $\zeta \in \mC$ a primitive $\ell$-th root of unity, with $\ell$ an odd integer, larger than the Coxeter number of any factor of $\fg$, and not divisible by 3 if $\fg$ contains a factor~$\tG_2$.
We let $G$ be the reductive group of adjoint type corresponding to $\fg$, acting faithfully on~$\fg$. 
We denote by
$$\mX^+\;\subset\;\mX\;\subset\;\fh^\ast$$
the lattice $\mX$ of weights appearing in finite dimensional $\fg$-modules and the subset $\mX^+$ of their highest weights. We denote by $\Phi\subset\mX$ the set of roots of $\fg$. The positive roots $\Phi^+$ correspond to the roots in $\fn^+$ and $\rho\in\mX$ is the half sum of elements in $\Phi^+$.

\subsubsection{}
Let $U_\zeta (\fg)$ be Lusztig's quantum group (a Hopf algebra over $\mC$) with divided powers, with the `small quantum group' Hopf subalgebra $\bu_\zeta(\fg)$. 
By a $U_\zeta(\fg)$-module, we refer to a finite-dimensional type 1 module. These modules form a ribbon tensor category $\Rep U_\zeta(\fg)$. We refer to \cite[Appendix~H]{Jantzen} for more details.

For a Levi subalgebra $\fl\subset\fg$, we set $\fl':=[\fl,\fl]$. Then we have canonical inclusions $U_\zeta(\fl')\subset U_{\zeta}(\fg)$ and $\bu_\zeta(\fl')\subset\bu_\zeta(\fg)$ as Hopf algebras.

\subsubsection{} We denote by $W_f$ the Weyl group of $\fg$, and by $W=W_f\ltimes Q$, with $Q$ the $\mZ$-span of $\Phi$ in $\mX$, the corresponding affine Weyl group. We denote generators of $W_f$ corresponding to the simple roots in $\Phi$ by $s_1,\ldots, s_r$ and the affine generator by $s_0$. The affine Weyl group acts on $\mX$ where
$$s_0(\lambda)\;=\;\lambda+(\ell-\vartheta^\vee(\lambda))\,\vartheta,$$
for $\vartheta\in\Phi$ the highest short root.
We write $W^+\subset W$ for the subset of shortest coset representatives of $W_f\backslash W$. Then for the $\rho$-shifted action, $W^+$ embeds into $\mX^+$ by $w\mapsto w\cdot 0=w(\rho)-\rho.$
For example, $s_0\cdot0=(\ell+1-h)\vartheta$.

The alcoves of $\mR\otimes \mX$ are the connected components of the complement of the reflection hyperplanes of $W$ under the $\rho$-shifted action. We are mainly interested in the intersection with $\mX$ of the lower closures of the alcoves. These are the subsets of the form
$$\{\lambda\in\mX\mid (n_\alpha-1)\ell\,\le\, \alpha^\vee(\lambda+\rho) \,<\,n_\alpha \ell,\;\forall \alpha\in\Phi^+\},$$
for fixed $n_\alpha\in\mZ$.

The representation category $\Rep U_\zeta(\fg)$ decomposes into a direct sum of the Serre subcategories generated by those simple modules with highest weights in one single $\rho$-shifted $W$-orbit, see \cite[Corollary~8.2]{APW}. The principal block $\Rep_0U_\zeta(\fg)$ is this subcategory corresponding to the orbit of $0$.

\subsubsection{Tilting modules} 
We denote by $\Tilt U_\zeta (\fg)$ the (ribbon) pseudo-tensor subcategory of $\Rep U_\zeta(\fg)$ of tilting modules, see \cite{An}. Indecomposable tilting modules are labelled by their highest weight as $T(\lambda)$, $\lambda\in\mX^+$. The category $\Rep U_\zeta(\fg)$ has enough projective objects, and all projective objects are tilting modules and also injective, see \cite{APW}. Concretely, the indecomposable projectives are $T((\ell-1)\rho+\mu)$, $\mu\in\mX^+$ and the projective cover, and injective hull, of $\unit$ is $T((2\ell-2)\rho)$.

For a Levi subalgebra $\fl\subset\fg$, the restriction (braided)  tensor functor from $\Rep U_\zeta(\fg)$ to $\Rep U_\zeta(\fl)$ maps tilting modules to tilting modules, see \cite[Corollary~1.9(i)]{Ka} for the integral version from which both the quantum and the algebraic group case follow.

\subsection{Nilpotent orbits}
We refer to \cite{CM} for details on nilpotent orbits.

\subsubsection{}Denote the nilpotent cone of $\fg$ by $\nil=\nil(\fg)$. We consider the set $\nil\hspace{-0.3mm}/G$ of non-empty $G$-orbits $\cO$ in $\nil$ as a partially ordered set, where $\cO_1$ is smaller than $\cO_2$ if $\cO_1$ is included in  the (Zariski) closure $\overline{\cO_2}$.

\subsubsection{Lusztig - Xi bijection}\label{LXbij} There are bijections
$$\nilo\;\longleftrightarrow\;\{\mbox{two-sided cells in $W$}\}\;\longleftrightarrow\;\{\mbox{right cells in $W^+$}\},$$
where the first one is from \cite{Lu4} and the second is given by intersection with $W^+$, see \cite[Theorem~1.2]{LX}. 
We denote the overall map from left to right by $\cO\mapsto c(\cO)$. Furthermore, we denote by $C(\cO)$ the union of all lower closures of the alcoves that contain an element $x\cdot 0$ for $x\in c(\cO)\subset W^+$. In particular $\mX^+=\sqcup_{\cO}C(\cO)$.

An orbit $\cO$ is {\bf distinguished} if its intersection with every proper Levi subalgebra of $\fg$ is empty.
By \cite[Theorem~8.1]{Lu4}, an orbit $\cO$ is distinguished if and only if $c(\cO)$ is finite.

\subsection{Support theory}\label{SecSupp}

\subsubsection{}\label{sec:GK}
As proved in \cite{GK}, see also \cite{BNPP}, we have an isomorphism
$$\Ext^{\bullet}_{\bu_\zeta(\fg)}(\unit,\unit)=\Ext^{2\bullet}_{\bu_\zeta(\fg)}(\unit,\unit)\;\simeq\;\mC^\bullet[\nil]$$
of graded algebras (for $\nil$ viewed as a cone). As can be verified from naturality of the construction in \cite{GK}, the isomorphisms admit a commutative square with restriction
\begin{equation*}\label{commdiaGK}
\xymatrix{
\Ext^{2\bullet}_{\bu_\zeta(\fg)}(\unit,\unit)\ar[d]\ar[rr]^{\sim}&&\mC^\bullet[\nil]\ar[d]\\
\Ext^{2\bullet}_{\bu_\zeta(\fl')}(\unit,\unit)\ar[rr]^{\sim}&&\mC^\bullet[\nil(\fl')],}
\end{equation*}
where the right map is restriction along $\nil(\fl')=\nil\cap\fl\hookrightarrow \nil$. 

\subsubsection{}For any $\bu_{\zeta}(\fg)$-module $M$, the kernel of the graded algebra morphism
$$\Ext^{\bullet}_{\bu_\zeta(\fg)}(\unit,\unit)\;\to\; \Ext^{\bullet}_{\bu_\zeta(\fg)}(M,M)$$
defines a subcone of $\nil$, its support variety (cone)
$\supp (M)\subset \nil.$ 
For example
$$\supp(\unit)=\nil,\quad \supp(P)=0,\quad \supp(0)=\varnothing,$$
if $P$ is projective.
If $M$ is actually a $U_\zeta(\fg)$-module, then the support variety is $G$-stable (this follows from the observation $\Ext^\bullet_{U_\zeta(\fg)}(U(\fg),-)\simeq \Ext^\bullet_{\bu_\zeta(\fg)}(\mC,-)$ in \cite[\S 5.2]{GK} and Remark (iii) on p180 of \cite{GK}) and hence a (finite) union of nilpotent orbits.

It follows from the naturality in \ref{sec:GK}, see \cite[Lemma~3.4(ii)]{OsSupp}, that
\begin{equation}\label{eq:IncSuppLevi}\supp \Res M\subset\nil(\fl)\cap\supp M,\end{equation}
which is expected to be an equality.

\subsubsection{}For $U_\zeta(\fg)$-modules $M,N$ we have
\begin{equation}\label{supptensor}
\supp(M\otimes N)\;\subset\;\supp(M)\cap\supp(N),
\end{equation}
see \cite[Lemma~3.6]{OsSupp}, while equality is expected. We will show in Proposition~\ref{prop:thickprime}(4) that at least equality holds for tilting modules. By \cite[Lemma~3.4]{OsSupp}, we also know that the support of a direct sum is the union of the supports of the summands.

For any $\lambda\in\mX^+$, we have
\begin{equation}\label{supptilt}\supp(T(\lambda))\;=\;\overline{\cO}\quad\Leftrightarrow \quad \lambda\in C(\cO).\end{equation}
Indeed, in type $\tA$ this is proved in \cite[Theorem~6.8]{OsSupp}. In general this was proved in \cite[Corollary~3]{Be}, for the support based on the $\mC[\nil]$-module $\Ext^\bullet_{\bu_{\zeta}(\fg)}(\unit,T(\lambda))$, and subsequently for the support notion as used above in \cite[\S 7.2]{BKN} and \cite[\S 8.4]{AHR1}.

\subsection{Hecke algebra} \label{SecHecke}
\subsubsection{}\label{ordKL} We consider the Hecke algebra corresponding to the affine Weyl group. Let $\H$ be the $\mZ[v,v^{-1}]$-algebra with basis $H_x, x\in W$ and multiplication given by $H_xH_y=H_{xy}$ if $\ell(xy)=\ell(x)+\ell(y)$ and, for each simple reflection $s$
$$H_s^2\;=\; (v^{-1}-v)H_s+1.$$

There exists a unique involutive automorphism $\overline{\cdot}$ of the ring $\H$ with $\overline{v}=v^{-1}$ and $\overline{H_x}=(H_{x^{-1}})^{-1}$. The Kazhdan-Lusztig basis of $\H$ is given by $\underline{H}_x$, where, for any $x\in W$, the element $\underline{H}_x$ is the unique element with
$$\overline{\underline{H}_x}=\underline{H}_x\quad\mbox{ and }\quad\underline{H}_x\;\in\; H_x+\sum_{y<x}v\mZ[v]H_y$$
for $\le$ the Bruhat order on $W$. We write $\underline{H}_x=\sum_{y}h_{y,x}H_y$ with $h_{y,x}\in\mZ[v]$.

\subsubsection{Anti-spherical module} The anti-spherical module for $\H$ is defined as the right module
$$\mathcal{AS}\;=\;\mZ[v,v^{-1}]\otimes_{\H_f}\H,$$
where $\H_f$ is the Hecke algebra for $W_f$, canonically a subalgebra in $\H$, and ${H}_{s_i}$ acts on $\mZ[v,v^{-1}]$ as $-v$ for $i>0$.
Then $1\otimes \underline{H}_{x}=0$ for $x\not\in W^+$ and
$$\{\underline{N}_x\;:=\; 1\otimes \underline{H}_x\mid x\in W^+\},$$
is a basis of the abelian group $\mathcal{AS}$. We also write $N_x=1\otimes H_x$, for $x\in W^+$ and define $n_{y,x}\in\mZ[v]$ by $\underline{N}_x=\sum_y n_{y,x}N_y$. Hence $n_{y,x}=0$ unless $y\le x$ and $n_{y,x}\in v\mZ[v]$ if $y<x$.

The following result does not seem to be available in the literature, but all the necessary ingredients can be found in \cite{Soergel}.
\begin{theorem}\label{thm:bound}
With $w_0$ the longest element of $W_f$, for all $x,y\in W^+$, we have
$$\deg n_{y,x}\;\le\;\ell(w_0)=\dim\fn^+.$$

\end{theorem}
\begin{proof}
Similarly to the anti-spherical module, there is the spherical $\H$-module, leading to a canonical basis $\underline{M}_x=\sum_{y}m_{y,x}M_y$, see \cite[\S 3]{Soergel}, for $m_{y,x}\in\mZ[v]$. As explained in the proof of \cite[Theorem~3.6]{Soergel}, these can be `inverted' via $m^{z,x}\in\mZ[v]$ (with $m^{z,x}=0$ unless $z\ge x$ and $m^{x,x}=1$) satisfying
$$\sum_z(-1)^{\ell(z)+\ell(x)}m^{z,x}m_{z,y}\;=\;\delta_{x,y}.$$
It then follows that also
$$\sum_x(-1)^{\ell(z)+\ell(x)}m^{z,x}m_{y,x}\;=\;\delta_{y,z}.$$

The main ingredient is \cite[Theorem~5.1]{Soergel}, which states that
$m^{y,x}= v^{\ell(w_0)} \overline{n_{y,\hat{x}}}$,
for $x\mapsto \hat{x}$ an operation on $W^+$. Though we will not use it, this already implies the statement in the theorem in most cases (for arbitrary $y$ and for $x$ in the image of $\,\hat{\cdot}\,$). More importantly for us, it shows that, for all $x,y$, the polynomials $m^{y,x}$ are bounded in degree by $\ell(w_0)$.

As shown in \cite[Theorem~3.5]{Soergel}, for the duality $\overline{\cdot}$ on $\mathcal{AS}$ (satisfying $\overline{mh}=\overline{m}\overline{h}$ for $m\in\mathcal{AS}$ and $h\in\H$), also the elements
$$\widetilde{\underline{N}}_x\;:=\; \sum_y (-1)^{\ell(x)+\ell(y)}\overline{m_{y,x}}N_y$$
are, like the $\underline{N}_x$, self-dual. Inverting the above formula shows that
$$\underline{N}_x\;=\;\sum_{y,z}n_{y,x}\overline{m^{y,z}}\,\widetilde{\underline{N}}_z.$$
By self-duality of $\underline{N}_x$ and $\widetilde{\underline{N}}_z$, the maximal positive degree of $v$ occurring with non-zero coefficient in the Laurent polynomial $\sum_{y}n_{y,x}\overline{m^{y,z}}\in\mZ[v,v^{-1}]$ is minus the minimal degree.

Now assume for a contradiction that $d:=\deg n_{z,x}>\ell(w_0)$ for some $x,z\in W^+$, and maximise this degree along $z$, while keeping $x$ fixed. Then the maximal (positive and thus also negative) degree in $\sum_{y}n_{y,x}\overline{m^{y,z}}$ is $d$, and hence there must be some $y\in W^+$ with $\deg m^{y,z}\ge d>\ell(w_0)$, a contradiction.
\end{proof}

\subsection{Graded lift}\label{sec:grli}

\subsubsection{}\label{desi} 

The principal block $\Tilt_0U_\zeta (\fg)$ has a graded lift. For our purposes this will be most conveniently interpreted as saying that every morphism space between $T_1,T_2\in\Tilt_0U_\zeta(\fg)$ has a canonical $\mZ$-grading
 $$\Hom(T_1,T_2)\;=\;\bigoplus_{i\in\mZ}\Hom(T_1,T_2)_i,$$ such that the composition of a morphism in degree $i$ and a morphism in degree $j$ produces a morphism in degree $i+j$.

To see this, we can use the Kazhdan-Lusztig equivalence, from \cite{KLq} in the simply-laced case and \cite[\S 8.4]{Lu-Mono} in the general case for large enough $\ell$, between $\Rep U_\zeta(\fg)$ and a negative level parabolic category $\mathcal{O}$ for the affine Kac-Moody algebra $\hat{\fg}$. By \cite{SVV} the latter has a graded lift. More concretely, finite truncations of $\mathcal{O}$ are module categories over finite dimensional algebras that admit a graded lift. The grading between truncations is compatible, so that it suffices to observe that tilting modules are gradable, see \cite[Proposition~2.7]{SVV} or \cite{MO}. The requirement that $\ell$ be large outside of simply-laced cases can then be dropped, since \cite[(1.1.2)]{ABG} shows that $\Tilt_0U_\zeta(\fg)$ does not depend on $\ell$. As a side note, the latter paper, in \cite[Corollary~9.2.4]{ABG}, actually also shows that there is a graded lift
$$\Psi:\;D^b(\mathsf{Coh}^{G\times\mG_m}(\tilde{\nil}))\;\to\; D^b(\Rep_0 U_\zeta(\fg))\supset \Tilt_0U_\zeta(\fg),$$
where $\tilde{\nil}\to\nil$ is the Springer resolution. 
For us, the main feature is compatibility with Kazhdan-Lusztig combinatorics:

\begin{prop}\label{Prop:GL}
For all $x,y\in W^+$, the grading from \ref{desi} satisfies
$$\sum_{i\in\mZ}v^i\dim \Hom(T(x\cdot0),T(y\cdot 0))_i\;=\;\sum_{z\in W^+}n_{z,x}(v)n_{z,y}(v)\;\in\mN[v].$$
\end{prop}
\begin{proof}
This is the graded refinement of the usual formula
$$\dim\Hom(T(x\cdot 0),T(y\cdot 0))\;=\;\sum_z(T(x\cdot 0):\Delta(z\cdot 0))\,(T(y\cdot 0):\nabla(z\cdot 0))$$
where, by \cite[Conjecture~7.1]{Soergel} and \cite{SoKM}:
$$(T(x\cdot 0):\Delta(z\cdot 0))=(T(x\cdot 0):\nabla(z\cdot 0))=n_{z,x}(1).$$

To formulate a rigorous argument for the version via $\hat{\fg}$, we consider the corresponding statement in the negative level parabolic category $\mathcal{O}$ for $\hat{\fg}$. The category of tilting modules there is, by the Ringel duality from \cite{SoKM} which is compatible with the Koszul grading as explained in \cite[\S 2.6]{SVV}, equivalent to the category of projective modules in a positive level parabolic category $\mathcal{O}$ for~$\hat{\fg}$. That the grading of the morphism spaces between projectives is given by the relevant parabolic Kazhdan-Lusztig polynomials is the content of \cite[Proposition~4.27]{SVV} thanks to the algebra isomorphism in \cite[Corollary~4.34]{SVV}.
\end{proof}

\section{Quantum groups: Tensor ideals}\label{sec:quantum2}
We keep $\bk=\mC$ and work with the conventions and assumptions from the previous section.

\subsection{Thick tensor ideals}
Thick tensor ideals in $\Tilt U_\zeta (\fg)$ were classified in \cite{Os}, see also \cite[Theorem~8.2.1]{BKN}.
 We choose a labelling system based on support theory, see \S \ref{SecSupp}.

 \begin{prop}\label{prop:thickprime}
 \begin{enumerate}
 \item The assignment 
 $$I\;\mapsto\;\{T\in\Tilt U_\zeta(\fg)\mid \supp(T)\subset I\}$$
 is an isomorphism between the poset of order-ideals in $\nilo$ and the poset of thick tensor ideals in $\Tilt U_\zeta(\fg)$. Here $\supp(T)\subset I$ is an abbreviation for $\supp(T)\subset \cup_{\cO\in I}\cO$.
 \item The assignment that sends $\cO\in\nilo$ to the set of tilting modules $T$ with $\supp(T)=\overline{\cO}$ is a bijection between $\nilo$ and the set of cells in $\Tilt U_\zeta(\fg)$.
 \item The assignment
$$\cO\;\mapsto\;\bI_\cO:=\{T\in\Tilt U_\zeta(\fg)\mid \cO\not\subset \supp(T)\}$$
is an isomorphism between the poset $\nilo$ and the poset of prime thick tensor ideals in $\Tilt U_\zeta (\fg)$.
\item For $T,T'\in \Tilt U_\zeta(\fg)$, we have
$$\supp(T\otimes T')\;=\;\supp(T)\cap\supp(T').$$
 \end{enumerate}
 \end{prop}

\begin{proof}
That the assignment in (1) takes values in thick tensor ideals follows from \eqref{supptensor}. That it is injective follows from \eqref{supptilt}. Also by \cite{OsSupp, Be},
 the thick tensor ideal generated by an indecomposable tilting module $T$ contains an indecomposable tilting module $T'$ if and only if $\supp(T')\subset\supp(T)$, see also \cite[Theorem~8.1.1(b)]{BKN}, so the assignment is surjective, and by construction inclusion preserving. Part (2) now follows immediately.

By construction, $\bI_{\cO}$ has a unique cover in the inclusion order on thick tensor ideals, namely the tensor ideal corresponding as in (1) to the ideal $I$ in $\nilo$ generated by $\cO$ and all orbits that are incomparable to $\cO$.
That $\bI_{\cO}$ is a prime thick tensor ideal thus follows from Lemma~\ref{LemPrimeCover}(1). That the assignment in (3) is injective follows from part (1). For an arbitrary thick tensor ideal $\bI$, let $\cO_1,\ldots,\cO_m$ be a complete list of nilpotent orbits that are not contained in the support of any $T\in \bI$. Then $\bI=\cap_{i}\bI_{\cO_i}$, by part (1).
 Since a prime tensor ideal cannot be the intersection of two tensor ideals (without being equal to one of them), the assignment in (3) is surjective too.
 
 For part (4), we already know the inclusion in one direction, see \eqref{supptensor}. The other inclusion can be argued as follows. If there were $T_1,T_2\in \Tilt U_\zeta(\fg)$ with $\supp(T)\supset\cO\subset \supp(T')$, but $\cO\not\subset\supp(T\otimes T')$, then we would have $T\otimes T'\in \bI_{\cO}$ while $T,T'\not\in\bI_{\cO}$, contradicting (3).
\end{proof}

\begin{remark}\label{Rem:PI}
\begin{enumerate}
\item As observed in the proof of Proposition~\ref{prop:thickprime}, arbitrary (proper) thick tensor ideals are intersections of the ideals $\bI_{\cO}$.
\item There is also a bijection between $\nilo$ and the set of non-zero thick tensor ideals in $\Tilt U_\zeta (\fg)$ generated by one indecomposable tilting module, given by 
$$\cO\;\mapsto\;\{T\in \Tilt U_\zeta(\fg)\mid \supp(T)\subset\overline{\cO}\}.$$
Even when the partial order on $\nilo$ is linear (and thus every proper thick tensor ideal is prime), this labelling convention is shifted from ours.
\end{enumerate}

\end{remark}

\subsection{Prime tensor ideals}

We abbreviate $\cI_{\cO}:=\bI^{\max}_{\cO}$. 
By Remark~\ref{RemNegSingle}, we know
$$\cN\;=\;\cI_{\cO_{\reg}}\;=\;\bI_{\cO_{\reg}}^{\max}\;=\;\bI_{\cO_{\reg}}^{\min}.$$

\begin{theorem}\label{ThmClassPrime}
The prime tensor ideals in $\Tilt U_\zeta (\fg)$ are precisely the ideals $\cI_{\cO}=\bI_{\cO}^{\max}$, with $\cO\in \nilo$.
\end{theorem}
\begin{proof}
By Theorem~\ref{Thm:FinR} and Proposition~\ref{prop:thickprime}(3), the conclusion follows from Lemma~\ref{Lem:N} below.
\end{proof}

\begin{lemma}\label{Lem:N}
We have $\cR\left(\Tilt U_\zeta(\fg)\right)^M=0$ for 
$M=2\ell(w_0)+1=\dim\nil+1.$
\end{lemma}
\begin{proof}
For the principal block $\Tilt_0U_\zeta(\fg)$, we can observe that the non-negativity of the grading implied by Proposition~\ref{Prop:GL} implies that $\cR(\Tilt_0U_\zeta(\fg))$ corresponds precisely to the morphism spaces of strictly positive degree. That the $M$-th power then vanishes then follows from the observation that the morphism spaces in degree $M$ and above are zero by Theorem~\ref{thm:bound}. For any other block $\cB$ in $\Tilt U_\zeta(\fg)$, we have a translation functor $\theta^{\out}:\cB\to \Tilt_0U_\zeta(\fg)$ with two-sided adjoint $\theta^{\on}$, see \cite[\S 8]{APW}. It suffices to observe that $\theta^{\out}$ is faithful and sends the radical into the radical. To prove these (known) statements, we consider the translation functors between the corresponding blocks of $\Rep U_\zeta(\fg)$. The first claim then follows for instance from exactness of $\theta^{\out}$ and the fact that every simple object in the block of $\cB$ is in the essential image of $\theta^{\on}$ by \cite[Theorem 8.3(ii)]{APW}. To prove the second claim, we consider a morphism in the radical $f:T_1\to T_2$ between indecomposable tilting modules in $\cB$. Then $\theta^{\out}(f):\theta^{\out}T_1\to\theta^{\out}T_2$ is again a morphism between indecomposable tilting modules, see \cite[Proposition~5.6]{An} for a special case and \cite[Proposition~5.2]{An3} in the more complicated modular setting. If $\theta^{\out}(f)$ is not in the radical then it is an isomorphism, which in particular implies that $T_1=T_2$, again by \cite[Proposition~5.6]{An}. However, we then get a contradiction with the fact that $f$, and so also $\theta^{\out}(f)$, is nilpotent.
\end{proof}

\subsection{The naturality conjecture}\label{sec:nat}

\begin{prop}\label{prop:naturality}
The following conditions are equivalent on a Levi subalgebra $\fl\subset\fg$:
\begin{enumerate}
\item For every orbit $\cO\subset\nil(\fl')=\nil(\fl)$, the preimage of $\bI_{\cO}$ under the restriction  
$$\Res:\;\Tilt U_\zeta (\mathfrak{g})\,\to\,\Tilt U_\zeta (\mathfrak{l'})$$ is given by $\bI_{G\cdot \cO}$.

\item For every $T\in \Tilt U_\zeta(\fg)$, we have
$\supp \Res T\;=\;\nil(\fl)\cap\supp T.$
\item For every $T\in \Tilt U_\zeta(\fg)$, we have
$\supp \Res T\;\supset\;\nil(\fl)\cap\supp T.$
\item For every $\cO\subset\nil(\fl)$, there is $T\in\Tilt U_\zeta (\fg)$ with $\supp T=\overline{G\cdot\cO}$ and $\cO\subset \supp\Res(T)$.
\end{enumerate}
\end{prop}
\begin{proof}
That (2) implies (3) is trivial. That (3) implies (2) follows from \eqref{eq:IncSuppLevi}.

We can also observe that conditions (1) and (2) are both equivalent to the condition
$$\cO\subset\supp \Res T\quad\Leftrightarrow\quad G\cdot\cO\subset \supp T.$$

That (4) is equivalent to the other conditions follows from the classification of thick tensor ideals in Proposition~\ref{prop:thickprime}. Concretely, if (2) is satisfied, then any $T$ with $\supp T=\overline{G\cdot\cO}$ (which exists by the classification) satisfies $\cO\subset\supp\Res(T)$. On the other hand, assume that (4) is satisfied. The preimage of $\bI_{\cO}$ contains $\bI_{G\cdot \cO}$, again by \eqref{eq:IncSuppLevi}. If it were strictly bigger then it would need to contain all tilting modules with $\supp T=\overline{G\cdot \cO}$. However, by assumption (4) at least one such tilting module is not in the ideal, a contradiction. Hence (1) follows.
\end{proof}

\begin{theorem}
If $\fg$ is of type $\tA$, the conditions in Proposition~\ref{prop:naturality} are satisfied.
\end{theorem}
\begin{proof}
We can reduce this to the case $\fg=\mathfrak{sl}_n$. Furthermore, by transitivity for (1), we can assume that $\cO$ is distinguished (and thus maximal).
Focusing on formulation (4), for $\lambda\vdash n$, we need to prove that there exists $T\in \Tilt U_\zeta (\fg)$ with $\supp(T)=\overline{\cO_\lambda}$ such that its restriction to the Levi subalgebra $\mathfrak{sl}_\lambda=\oplus_i\mathfrak{sl}_{\lambda_i}$,
$$\Tilt U_\zeta (\mathfrak{sl}_n)\;\to\; \Tilt U_\zeta (\mathfrak{sl}_\lambda)$$
is not negligible (the statement for conjugate Levi subalgebras then also follows). We can take the (simple) tilting module with desired support constructed in \cite[\S 6]{OsSupp}. Indeed, one can take $T(\kappa)$, for a weight $\kappa$ with $\langle\kappa,\alpha_i^\vee\rangle=0$ whenever $i$ is not equal to a partial sum $\sum_{a=1}^b\lambda_a$, so that by a standard argument its restriction to $U_{\zeta}(\mathfrak{sl}_{\lambda})$ has a summand $\unit$.
\end{proof}

Since Proposition~\ref{prop:naturality}(2) is even expected to be true for arbitrary modules, we conjecture:

\begin{conjecture}\label{conj:nat}
The conditions in Proposition~\ref{prop:naturality} are satisfied for an arbitrary $\fg$ (with assumptions as in \ref{sec:condl}), i.e. we have naturality of support of tilting modules with respect to Levi subalgebras.
\end{conjecture}

\begin{remark}Since the corresponding naturality of support is known for reductive groups, for $\ell$ a prime, it would suffice to find in each cell a quantum tilting module that has the same character as the corresponding tilting module for the reductive group in characteristic~$\ell$. This aligns with our proof for type $\tA$.
\end{remark}

\subsection{A conditional classification}

\begin{theorem}\label{thm:distab}
If $\cO$ is a distinguished orbit, then $\Tilt U_\zeta (\fg)/\cI_{\cO}$ admits an abelian envelope, in particular $\cI_{\cO}$ is of abelian type.
\end{theorem}
\begin{proof}
We can apply \Cref{LemStr}(2). Indeed, we have already observed that the cover $\bI_{\cO}<\bJ$ exists, see also~Lemma~\ref{LemPrimeCover}(2). Furthermore, the indecomposable objects in $\bJ\,\backslash\, \bI_{\cO}$ are precisely $T(\lambda)$, $\lambda\in\mX^+$, with
$\supp(T(\lambda))= \overline{\cO}.$
In other words, $\lambda$ belongs to the lower closure of an alcove labelled by $w\in c(\cO)$. But $c(\cO)$ is finite by \ref{LXbij}, so the conditions of the theorem are satisfied.
\end{proof}

\begin{example}
If $\cO=\cO_{\reg}$, then $\cI_{\cO_{\reg}}=\cN$ and 
$\Tilt U_\zeta (\fg)/\cN$, usually denoted as $\Ver_\zeta (\fg)$ (the Verlinde category) is itself (semisimple) abelian.
\end{example}

\begin{theorem}\label{ThmAb}
Assume that Conjecture~\ref{conj:nat} is satisfied for $\fg$, e.g. $\fg$ is of type $\tA$. The tensor ideals of abelian type in $\Tilt U_\zeta(\fg)$ are in bijection with nilpotent orbits, so they are precisely the ideals $\cI_{\cO}=\bI_{\cO}^{\max}$, for $\cO\in\nilo$, and every map in \eqref{themap} is a bijection.
\end{theorem}
\begin{proof}
Consider a minimal Levi subalgebra $\fl\subset\fg$ that intersects $\cO$ non-trivially. Then any orbit $\cO'$  in $\cO\cap \fl$ is a distinguished orbit in $\fl'$ and $\cO=G\cdot \cO'$.

The inverse image $\cJ$ in $\Tilt U_\zeta(\fg)$ of $\cI_{\cO'}$ under restriction is of abelian type by Theorem~\ref{thm:distab}. Assuming Conjecture~\ref{conj:nat}, we have $\Ob(\cJ)=\bI_{\cO}$, so that Theorem~\ref{ThmClassPrime} implies $\cJ=\cI_{\cO}$. Hence all the prime tensor ideals must be of abelian type.
\end{proof}

\begin{remark}
\begin{enumerate}
\item Our results show that the inverse image of a prime tensor ideal along restriction $\Tilt U_\zeta(\fg)\to\Tilt U_\zeta(\fl')$ only depends on the conjugacy class of $\fl\subset \fg$. Indeed, the corresponding claim is true for thick tensor ideals by Weyl group invariance of characters. The result is then carried over to tensor ideals by the bijection in Theorem~\ref{ThmClassPrime}.
\item One can also consider the `mixed' case (quantum groups in positive characteristic) of the problems we consider, see \cite{Dc} for the abelian envelopes in that generality for~$SL_2$.
\end{enumerate}

\end{remark}

\subsubsection{} We did not only prove that the prime ideal $\cI_{\cO}$ in $\Tilt U_\zeta(\fg)$ corresponding to a distinguished orbit is of abelian type, Theorem~\ref{ThmAbEnv} implies that $(\Tilt U_\zeta(\fg))/\cI_{\cO}$ then even has an abelian envelope $\cC(\fg,\ell,\cO)$.
Similarly, an affirmative answer to the following question would imply that such an abelian envelope $\cC(\fg,\ell,\cO)$ exists for all orbits $\cO$ and in particular provide an unconditional proof of Theorem~\ref{ThmAb}:
\begin{question}
Is the thick tensor ideal of tilting modules $T$ with $\supp T\subset\overline{\cO}$  inside $\Tilt U_\zeta(\fg)/\cI_{\cO}$ a splitting ideal?
\end{question}

\subsection{Connection with Duflo involutions}

\subsubsection{} For each simple reflection $s$ in the affine Weyl group, we have an associated translation functor $\theta_s$ through the corresponding wall (the composite of translating onto and out of the wall as in \cite[\S 8]{APW}). On the Grothendieck group this action corresponds to the action of $\mZ W$ ($\H$ specialised at $v\mapsto 1$) on the antispherical representation, where the action of $\theta_s$ corresponds to acting with $1+s$.

Following \cite{RW}, we can define a diagrammatic Hecke category $\cD$ over $\bk$, which is analogous to but different from the category of Soergel bimodules corresponding to $W$ as considered in \S\ref{SecSB} in the appendix. The Grothendieck ring is again $\H$, and it is a category $\cS$ as in \ref{DefS} (even satisfying \ref{Piv}, see \cite[\S 4]{RW} although we do not need that) with $\mW=W$, where the necessary results from \cite{So} now have to be replaced with their generalisations in \cite{EW16}. In particular, its theory of Duflo involutions matches the classical theory of Duflo involutions of $W$, and we use $\sD$ unambiguously for the set of Duflo involutions.

Moreover, as in \cite[\S 4.4]{RW}, we can define the anti-spherical category $\cD^{\asph}$ as the quotient of $\cD$ with respect to the right tensor ideal generated by the objects $B_{w}$ with $w\not\in W^+$. Its split Grothendieck group is the antispherical module, with action of $\H$ coming from the right module structure of $\cD$ on $\cD^{\asph}$.

\subsubsection{}\label{hypoRW}

Following \cite[Conjecture~5.1]{RW}, we conjecture that the assignment of $B_s$ to a wall-crossing functor $\theta_s$ for all $s$ induces a right action of the Hecke category $\cD$ on $\Rep_0U_\zeta(\fg)$. As a consequence, see \cite[Theorem~1.2]{RW}, the action of $\cD$ on $\Rep_0 U_\zeta(\fg)$ then induces an equivalence
\begin{equation}\label{eqRW}
F:\;\cD_{\deg}^{\asph}\;\xrightarrow{\sim}\;\Tilt_0U_\zeta(\fg),\qquad B_w\mapsto T(w\cdot 0).
\end{equation}
Here $\cD_{\deg}^{\asph}$ stands for the degrading of $\cD^{\asph}$. It has the same objects, but the morphism spaces are given by $\oplus_i\cD^{\asph}(X,Y\langle i\rangle)$.

\subsubsection{} Equivalence~\eqref{eqRW} sends the monoidal unit in $\cD_{\deg}$ to the monoidal unit in $\Tilt U_\zeta(\fg)$, so that it follows by dual application of Theorem~\ref{ThmRigBij} that there is a canonical bijection between right tensor ideals in $\cD_{\deg}$ that contain all $B_w$ for $w\not\in W^+$ and tensor ideals in $\Tilt U_\zeta(\fg)$. One can show that two tensor ideals in $\cD_{\deg}$ have the same underlying thick tensor ideal if and only if the corresponding ideals in $\Tilt U_\zeta(\fg)$ have the same underlying thick tensor ideal. Most of the argument is given by Proposition~\ref{prop:Duflo} below. With some more work one can express some of the prime condition for tensor ideals in $\Tilt U_\zeta(\fg)$ on the $\cD_{\deg}$-side, so that Lemma~\ref{LemAlt} leads to an alternative proof of Theorem~\ref{ThmClassPrime}. Since this does not give additional results, we omit the details of the latter.

\begin{prop}\label{prop:Duflo}
Assume that Hypothesis \ref{hypoRW} is valid. Then for each $d\in\sD\cap W^+$, there exists a unique, up to post-composition with an isomorphism, $\alpha_d:\unit\to T(d\cdot 0)$ such that for each indecomposable tilting module $T$, there is a unique $d\in\sD$ so that $\co_T$ factors through $\alpha_d$ via a summand $T(d\cdot 0)$ of $T\otimes T^\ast$. Moreover, if $\supp(T)=\overline{\cO}$ then $d$ is the unique element of $c(\cO)\cap\sD$.
\end{prop}
\begin{proof}
Note that uniqueness is immediate from the stated properties. We show that we can take $\alpha_d$ to be the image of the distinguished morphisms in $\cD$ from Definition~\ref{Def:dist}.

We claim that for any two $X,Y\in\cD_{\deg}$, there exists a tilting module $T$ that contains both $F(Y\otimes X^\ast)$ and $FY\otimes (FX)^\ast$ as direct summands such that the inclusions yield a commutative diagram
$$
\xymatrix{
\cD_{\deg}^{\asph}(X,Y)\ar[r]^-{\sim}\ar[d]^\sim&\cD^{\asph}_{\deg}(\unit, Y\otimes X^\ast)\ar[r]^-{\sim}&\Hom(\unit, F(Y\otimes X^\ast))\ar[d]\\
\Hom(FX,FY)\ar[r]^-{\sim}&\Hom(\unit,FY\otimes (FX)^\ast)\ar[r]&\Hom(\unit, T).
}$$
The isomorphisms are either evaluation of $F$, or adjunction (where the relevant adjunction in $\cD^{\asph}$ is inherited from the monoidal structure on $\cD$ by taking a quotient with respect to a right tensor ideal).
Indeed, it suffices to prove the claim for $X$ of the form $B_{1}\otimes B_2\otimes \cdots\otimes B_i$, where each $B_i$ is a generator $B_{s_j}$. With $U=F(Y)$, we can then identify $F(Y\otimes X^\ast)$ with $\theta_i\theta_{i-1}\cdots \theta_1(U)$ and $FY\otimes (FX)^\ast$ with  $U\otimes \theta_i\theta_{i-1}\cdots \theta_1(\unit)$.
Each functor $\theta_i$ is a composite of projections onto blocks and taking tensor products with tilting modules. By removing all such projections from both above expressions, we get the desired $T$. Commutativity then follows by observing that adjunction for translation functors is inherited from monoidal adjunction.

We apply this principle to $X=Y=B_x$. It then follows that the summand $B_d$ in $B_x\otimes B_x^\ast$ from Theorem~\ref{ThmP3} must lead to a summand $F(B_d)=T(d\cdot 0)$ in $T$ that is also a summand of $T(x\cdot0)\otimes T(x\cdot 0)^\ast$ through which $\co_{T(x\cdot 0)}$ factors, concluding the proof.
\end{proof}

\begin{remark}
In \cite{HW}, a combinatorial description of certain thick tensor ideals in $\Tilt U_\zeta(\fg)$ is given. Also a chain $\bN=\bN_1\supset \bN_2\supset\cdots$ of generalisations of the thick tensor ideal of negligible objects is constructed. In \cite[\S 10.4]{HW} an interpretation in terms of Lusztig's a-function is suggested, and proved in type $\tA$ in \cite[\S 9.6]{HW}. If correct, then our results would imply that the indecomposable objects in $\bN_i$ are those for which the coevaluation factors through an indecomposable direct summand with a morphism which comes from a morphism in degree $\ge i$ in $\cD^{\asph}$.
\end{remark}


\section{Quantum groups: Rank 2 cases}\label{sec:quantum3}

In this section, we make our results on $\Tilt U_\zeta(\fg)$ more explicit for $\fg$ of rank 2, and place them in the context of all tensor ideals.

\subsection{Classification of all tensor ideals}

In order to classify all tensor ideals, one needs to understand the space $\oplus_x\Hom(\unit,T(x\cdot 0))$ and its module structure. 
For simple Lie algebras of rank 2, the anti-spherical Kazhdan-Lusztig polynomials, and so in particular the dimensions of these spaces, have been computed explicitly in~\cite{Stroppel}. In general, it is clear that we can actually restrict the labelling by $x\in W^+$ significantly:
\begin{lemma}\label{xinv}
If for $x\in W^+$, we have a monomorphism $\unit\to T(x\cdot 0)$ then $x$ is a shortest representative in $W_f\backslash W/ W_f$. 
\end{lemma}
\begin{proof}
If $x$ is not shortest in $W/W_f$ then $x=ys_i$ for a simple $s\in W_f$ and $\ell(x)>\ell(y)$, and
$$\Hom(\unit, T(ys\cdot 0))\subset\Hom(\unit, \theta_{s}T(y\cdot 0))\;=\;\Hom(\theta_{s}\unit, T(y\cdot 0))=0,$$
since $\theta_{s}\unit=0$.
\end{proof}

The following elementary observation can help to understand the module structure.

\begin{lemma}\label{LemCO}Let $\cA$ be a pseudo-tensor category and consider $X\in \cA$ and a morphism $\alpha:\unit\to Y$ in $\cA$ such that $\co_X:\unit\to X\otimes X^\ast$ factors via $\alpha$. Then any morphism $\unit\to Z$ to any direct summand $Z$ of $X\otimes X^\ast$ factors through $\alpha$.
\end{lemma}
\begin{proof}
Any morphism $f:\unit\to X\otimes X^\ast$ factors through $\alpha$, since by adjunction $f$ can always be written as
$(\phi\otimes X^\ast)\circ \co_X,$
for an endomorphism $\phi$ of $X$.
\end{proof}

\subsubsection{Notation}

For $\fg$ simple (or more specifically $\fg=\mathfrak{sl}_n$), we avoid the use of double subscripts with standard orbits like $\cO_{\reg}$ or $\cO_{\zero}$ (or $\cO_\lambda$ for $\lambda\vdash n$), by writing
$$\bI_{{\mathrm{reg}}}\;=\; \bN\;=\; \bI_{(n)},\quad \bI_{{\mathrm{min}}}\;=\; \bP\;=\;\bI_{(2,1,\ldots, 1)} \quad\mbox{and}\quad \bI_{{\mathrm{zero}}}\;=\;\bo\;=\;\bI_{(1,\ldots, 1)}.$$
Here $\bP$ is the thick ideal of projective modules. By \cite[Theorem~9.12]{APW} it is generated by~$T((\ell-1)\rho)$.

\subsection{The example $\tA_2$}
Here we consider $\fg=\mathfrak{sl}_3$. 

\subsubsection{}\label{indecA2} We set
$$T_0:=T(0)=\unit,\quad T_1:=T(s_0\cdot 0)=T((\ell-2)\rho),$$
$$T_2^L:=T(s_0s_2s_1s_0\cdot 0)=T((3\ell-3)\omega_1),\quad T_2^R:=T(s_0s_1s_2s_0\cdot 0)=T((3\ell-3)\omega_2),$$
where $\omega_i\in\mX$ are the fundamental weights and $L,R$ just refers to left and right. Finally, we denote the injective hull of $\unit$ by
$$T_3\;=\; T(s_0s_1s_2s_1s_0\cdot0)\;=\; T((2\ell-2)\rho).$$

By \cite{Stroppel}, the morphism space from $\unit$ into an indecomposable tilting module $T$ is either zero, or one-dimensional. The latter is precisely the case if and only if $T$ is in the above list. Moreover, $\unit\hookrightarrow T_i^{\dagger}$, for $0\le i\le 3$ (and $\dagger$ empty or $L,R$) is in degree $i$ in the graded lift from Section~\ref{sec:grli}.

\subsubsection{} As we know that there is a unique maximal ideal and that $T_3$ is the injective hull of~$\unit$, the morphisms factor as
$$\xymatrixrowsep{1.8mm}
\xymatrix{
&&T_{2}^L\ar[rd]\\
\unit\ar[r]& T_{1}\ar[ru]\ar[rd]&& T_{3}.\\
&&T_{2}^R\ar[ru]}$$
We can thus complete our classification $\{\cI_{(3)},\cI_{(2,1)},\cI_{(1,1,1)}\}$ of tensor ideals of abelian type into a full classification.

\begin{prop}\label{PropLatticeSL3}
The lattice of proper tensor ideals in $\Tilt U_\zeta (\mathfrak{sl}_3)$ is given by
$$\xymatrixrowsep{3.1mm}\xymatrix{
&\cN=\cI_{(3)}\ar@{-}[d]&\\\
&\cI_{(2,1)}\\
\cJ^L\ar@{-}[ru]&&\cJ^R\ar@{-}[lu]\\
&\cJ\ar@{-}[ru]\ar@{-}[lu]\\
&0=\cI_{(1,1,1)}\ar@{-}[u]
}$$
where $\cJ^L$ resp. $\cJ^R$ is generated by $\unit\to T_{2}^L$, resp. $\unit\to T_2^R$, and $\cJ=\bI_{(2,1)}^{\min}=\cJ^L\cap\cJ^R$ is generated by $\unit\to T_3$.
\end{prop}

\subsection{Example: $\tB_2$} In this section we consider $\fg=\mathfrak{so}_5$. We choose a basis $(1,0)$, $(0,1)$ of $\mX$, so that the simple positive roots are $(1,-1)$ and $(0,1)$. The proper thick tensor ideals are
$$0=\bI_{{\zero}}\;\subset\; \bP=\bI_{{\min}}\;\subset\;  \bI_{\subreg}\;\subset\; \bI_{{\reg}}=\bN.$$

\subsubsection{}We set
$$T_0=T(0,0)=\unit, \quad T_1=T(\ell-3,0),\quad T_2=T(2\ell-3,\ell-1),$$
$$ T_3=T(2\ell-2,2\ell-2),\quad T_3^{[i]}=T(i\ell-3,0),\quad T_4=T(3\ell-3,\ell-1),$$
for $i\in\mZ_{\ge 3}.$
By \cite{Stroppel}, $\dim \Hom(\unit, T(\lambda))\le 1$, and only non-zero if $T(\lambda)$ is in the above list. Moreover, the non-zero morphism $\unit\to T_j^{\dagger}$ is contained in degree $j$ (for $\dagger$ empty or $[i]$).

Contrary to type $\tA_2$ (and contrary to the restriction to prime ideals), there are infinitely many tensor ideals. We quantify this infinite behaviour (in `depth' and `breadth') with respect to the inclusion order as follows.

\begin{prop}\label{prop:B2}
\begin{enumerate}
\item Tensor ideals in $\Tilt U_\zeta(\fg)$ do not satisfy the ascending or descending chain condition. 
\item The poset of tensor ideals is not interval finite.
\item There are infinite sets of mutually incomparable tensor ideals in $\Tilt U_\zeta(\fg)$.
\end{enumerate}
\end{prop}
\begin{proof}
Let $\cI_j$ be the tensor ideal generated by all $\unit\to T_3^{[i]}$ for $i\ge j$, let $\cI_j'$ be the tensor ideal generated by all $\unit\to T_3^{[i]}$ for $i\le j$. Let $\cJ$ be the tensor ideal generated by $\unit\to T_4$. Using the grading, and the principle in Theorem~\ref{ThmRigBij}, it follows that we have strict inclusions
$$\cI_3\supset \cI_4\supset \cI_5\supset\cdots\supset \cJ\quad\mbox{and}\quad \cI_3'\subset\cI'_4\subset\cI'_5\subset\cdots,$$
which proves parts (1) and (2). Part (3) follows similarly by taking infinite collections of incomparable subsets of $\mZ_{\ge 3}$.
\end{proof}

\subsubsection{} We have
$$T_1, T_3^{[i]}\in \bI_{\reg}\backslash \bI_{\subreg},\quad T_2,T_3\in \bI_{\subreg}\backslash \bI_{\min},\quad T_4\in \bI_{\min}=\bP.$$
It follows easily that the morphisms from $\unit\to T_1,$ $\unit \to T_2$ and $\unit \to T_4$ are the generators of $\cN$, $\bI^{\min}_{\subreg}$ and $\bI_{\min}^{\min}$ respectively. To determine the structure of the tensor ideals it suffices to establish whether $\unit\to T_3^{\dagger}$ factors via $T_2$ or only via $T_1$. For example:

\begin{lemma}
The morphism $\unit\to T_3$ factors via $T_2$.
\end{lemma}
\begin{proof}
Since $T(\ell-1,\ell-1)$ is self-dual and belongs to $\bI_{\subreg}$, the  conclusion follows from \Cref{LemCO}.
\end{proof}

\subsection{Example: $\tG_2$}

\subsubsection{}We embed the weight space of $\fg$ into three-dimensional space, so that the simple positive roots are
$(1,-1,0)$ and $(-1,2,-1)$, the highest root is $(1,1,-2)$ and the highest short root is $(1,0,-1)$. The proper thick tensor ideals are
$$0=\bI_{{\zero}}\;\subset\; \bP=\bI_{{\min}}\;\subset\; \bI_{{\supmin}}\;\subset\; \bI_{\subreg}\;\subset\; \bI_{{\reg}}=\bN.$$

\subsubsection{} We label again some tilting modules, and denote the maximal thick tensor ideal they belong to. We consider $T_0=\unit$ and
$$T_1=T(\ell-5,0,5-\ell)\in\bI_{{\reg}}, \quad T_2=T(2\ell-4,\ell-2,6-3\ell)\in\bI_{{\subreg}}$$
$$T_3=T(3\ell-5,0,5-3\ell)\in \bI_{{\supmin}},\quad T_4=T(4\ell-4,\ell-2,6-5\ell)\in \bI_{\supmin},$$
$$T_5^L=T(3\ell-3,3\ell-3,6-6\ell)\in \bI_{\subreg},\quad T_5^S=T(5\ell-5,0,5-5\ell)\in\bI_{{\supmin}},$$
$$T_6=T((2\ell-2)\rho)=T(4\ell-4,2\ell-2,6-6\ell)\in \bI_{{\min}}.$$
Here $L,S$ refers to `long' and `short'. By \cite{Stroppel}, $\Hom(\unit, T_i^\dagger)$ is one-dimensional and contained in degree $i$, for $0\le i\le 6$. Morphisms from $\unit$ to other indecomposable tilting modules are zero.

It seems a difficult problem to determine precisely which morphisms factor through which, although some cases follow from the general theory. A plausible picture for the poset of tensor ideals, with the prime tensor ideals $\cI_{\cO}$ depicted by black dots is as follows:
$$\xymatrixrowsep{1.6mm}
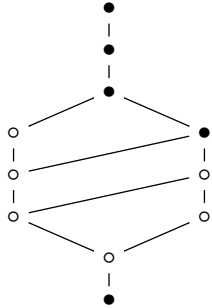
\xymatrix{
&\bullet\ar@{-}[d]&\\
&\bullet\ar@{-}[d]&\\
&\bullet\\
\circ\ar@{-}[ru]&&\bullet\ar@{-}[lu]\\
\circ\ar@{-}[u]\ar@{-}[urr]&&\circ\ar@{-}[u]\\
\circ\ar@{-}[u]\ar@{-}[urr]&&\circ\ar@{-}[u]\\
&\circ\ar@{-}[ru]\ar@{-}[lu]\ar@{-}[d]\\
&\bullet
}$$   


\section{Further cases of interest}\label{sec:further}

\subsection{The Oriented Brauer category}\label{sec:RepGLt}

Let $\delta\in\bk$. Deligne's category $(\Rep GL)_\delta$ is the universal symmetric pseudo-tensor category on one object of categorical dimension $\delta$, see~\cite{De}. It is the pseudo-abelian envelope of the oriented Brauer category $\OB(\delta)$.

If $\mathrm{char}(\bk)=0$, then $(\Rep GL)_\delta$ has non-zero tensor ideals if and only if $\delta\in\mZ\subset\bk$. It is then proved in \cite{Selecta} that \eqref{ObAll} is a bijection, and all tensor ideals are of abelian type.

For the rest of the section we assume that $\mathrm{char}(\bk)=p>0$. It is then similarly proved in \cite[Corollary~6.3.3]{Sinf} that $(\Rep GL)_\delta$ only has proper non-zero tensor ideals if $\delta\in\mF_p\subset\bk$.

We do not expect the following example to be specific to $p=2$. In fact, some conjectures in \cite{Sinf} predict that most faithfully prime tensor ideals in $(\Rep GL)_\delta$ will not be semiprimitive.
\begin{example}\label{ex:new}
If $p=2$, there is a faithfully prime tensor ideal in $(\Rep GL)_0$ that is not semiprimitive. Let $P$ be the projective cover of $\unit$ in $\Ver_4^+$, see \cite{BEO}, which is a self-extension of $\unit$. Then we can consider kernels of the symmetric tensor functors $(\Rep GL)_0\to\Ver_4^+$ that send the generator to $P$ and to $\unit^2$. By construction they have the same underlying thick tensor ideal. However, the tensor ideals are different, see either \cite[Theorem~5.3.1]{Sinf} or \cite[Example~4.6]{CEO}. We refer to \cite[Remark~2.4.7]{Sinf} for more details.
\end{example}

In other words, the diagonal map in \eqref{themap} is surjective but need not be a bijection.
A negative answer to the first question below would disprove the surjectivity in \cite[Conjecture~5.1.2]{Sinf}. To indicate the difficulty, this question includes the corresponding question for $\Tilt GL_n$ for $n\ge p$.

\begin{question} Let $\delta\in\mF_p$.
\begin{enumerate}
\item Is every faithfully prime tensor ideal in $(\Rep GL)_\delta$ of abelian type? 
\item Is every prime ideal in $(\Rep GL)_\delta$ faithfully prime?
\end{enumerate}

\end{question}

To complete the proof of \cite[Theorem~6.3.2]{Sinf}, we point out the following.

\begin{lemma}
Assume $p>3$. There are tensor ideals in $\OB(3)$ that are not prime.
\end{lemma}
\begin{proof}
It suffices to consider the preimage of the tensor ideal in $\Tilt SL_3$ from Lemma~\ref{lem:sl3p} below, under the universal tensor functor from $\OB(3)$. Here we can use that the tensor ideal theory of $\OB(3)$ is identical to that of $(\Rep GL)_3$ and that every morphism in $\Tilt SL_3$ is in the image of $(\Rep GL)_3\to\Tilt SL_3$.
\end{proof}

\subsection{Reductive groups in positive characteristic}
As explained in the introduction, the motivating problem for the current work is the classification of the tensor ideals of abelian type in $\Tilt G$, where $G$ is a reductive group over an algebraically closed field $\bk$ of characteristic $p>0$. Only the rank-one case is currently understood, see \cite{Selecta}.

A conjecture for the classification of thick tensor ideals (or at least the cells) in $\Tilt G$, for $p$ larger than the Coxeter number for $G$, is formulated in \cite{AHR2}. If correct, then every thick tensor ideal can be realised as the inverse image of a tensor ideal in $\Tilt(SL_2^{\times s})$, for some $s\in\mZ_{>0}$, under a restriction functor. If we make the extra assumption that only semiprimitive tensor ideals are faithfully prime, see Question~\ref{propmodular}, then it would follow that $\Upsilon$ in \eqref{themap} is a surjection. More precisely, every tensor ideal of abelian type in $\Tilt G$ would be the kernel of a symmetric tensor functor to $\Ver_{p^\infty}$, in line with \cite[Conjecture~1.4]{BEO}.

While the decisive features of the grading on $\Tilt_0 U_\zeta(\fg)$, that it is non-negative and bounded, are not satisfied for $\Tilt_0G$, for the moment we have no counterexamples to the following question:
\begin{question}\label{propmodular}
Let $G$ be a reductive group with Coxeter number below $p$. Are the following conditions equivalent on a tensor ideal $\cI$ in $\Tilt G$?
\begin{enumerate}
\item $\cI$ is semiprimitive and $\Ob(\cI)$ is prime;
\item $\cI$ is faithfully prime;
\item $\cI$ is prime.
\end{enumerate}
\end{question}
Note that,  if $p=2$, (1) and (2) are not equivalent for the category of tilting modules of $GL_{2m}$ (with $m>1$) that are direct summands of tensor products of copies of the natural representation and its dual, as follows from Example~\ref{ex:new}. For the moment such behaviour seems restricted to characteristics below the Coxeter number.

\begin{prop}\label{ThmSL3}
Assume $p>3$. Every thick tensor ideal in $\Tilt SL_3$ is prime, and $\Upsilon$ in \eqref{themap} is surjective. Concretely, we have a chain of tensor ideals lying above the thick tensor ideals from \cite[Example~15]{An2}
$$\Tilt SL_3\;\supset\; \cI_1 \;\supset\; \cI_2\;\supset\;\cI_3\;\supset\;\cdots,$$
where $\cI_{2n-1}$ and $\cI_{2n}$ are kernels of symmetric tensor functors $\Tilt SL_3\to \Ver_{p^n}$.
\end{prop}
\begin{proof}
The thick tensor ideals in $\Tilt SL_3$ form one descending chain $\bI_{i-1}\supset\bI_{i}$, by \cite[Example~15]{An2}.
The tensor ideals in $\Tilt SL_2$ are given as
$$\Tilt SL_2\;\supset\; \cJ_1 \;\supset\; \cJ_2\;\supset\;\cJ_3\;\supset\;\cdots$$
and the abelian envelope of $\Tilt SL_2/\cJ_n$ is $\Ver_{p^n}$, see \cite{BEO, AbEnv}. We have $\cJ_i=\bJ_i^{\min}=\bJ^{\max}_i$ for the thick tensor ideal $\bJ_i$ of tilting modules $T(m)$ with $m\ge p^i-1$.  
It thus suffices to demonstrate that, under restriction to the principal $SL_2$ subgroup, we have
$$\Res_P:\Tilt SL_3\to \Tilt SL_2,\quad \Res^{-1}_P(\bJ_n)\;=\; \bI_{2n-1},$$
while for a Levi subgroup $SL_2$, we have
$$\Res_L:\Tilt SL_3\to \Tilt SL_2,\quad \Res^{-1}_L(\bJ_n)\;=\; \bI_{2n}.$$
We can observe that the case $n=1$ follows from naturality of support and Humphreys conjecture, see \cite{Ha}.

We let $\rho,\omega_1,\omega_2$ refer to the half-sum of roots and the two fundamental weights of $SL_3$ and integers to $SL_2$-weights.
With $St_n:=T((p^n-1)\rho)$, we have that $\Res_PSt_n$ is $T(2p^n-2)$ up to summands with lower label and hence $\Res_PSt_n\not\in \bJ_{n+1}$. On the other hand, by Donkin's tensor product formula of \cite{Do}, for $m,n\in\mN$,
$$T[n]:=T((p^n-1)\rho+p^nm\omega_2)\;=\; St_n\otimes T(m\omega_2)^{(n)},$$
where $St_n$ is itself a product of $St_1^{(i)}$ for $0\le i<n$. By Donkin's formula, for neglibile $SL_2$-tilting modules $T_0,T_1\ldots, T_n$, the module $T_0\otimes T_1^{(1)}\otimes\cdots\otimes T_n^{(n)}$ is a tilting module in $\bJ_{n+1}$. Now $\Res_P St_1$ and $\Res_P T(m\omega_2)$, if we now assume $m\ge p-2$, are negligible, by the last sentence of the previous paragraph, so $\Res_P T[n]\in \bJ_{n+1}$. The claim for $\Res_P$ now follows.

Using similar arguments one can observe that $\Res_L St_n\in \bJ_{n}$ (by Donkin's formula) and $\Res_LT[n-1]\not\in\bJ_n$ (by highest weight) if we consider the Levi subgroup $SL_2<SL_3$ where the positive root of $SL_2$ is orthogonal to $\omega_2$. This proves the statement for $\Res_L$.
\end{proof}

Abelian versions of $\Tilt SL_3/\cI_{2i-1}$, see Theorem~\ref{ThmAbEnv}(2), as well as analogues for other simple groups~$G$, will be constructed and studied in forthcoming work of Newton in \cite{Joe}.

%

\begin{lemma}\label{lem:sl3p}
Assume $p>3$. There are tensor ideals in $\Tilt SL_3$ that are not prime.
\end{lemma}
\begin{proof}
We can mimic the non-prime tensor ideal $\cJ$ from Proposition~\ref{PropLatticeSL3}. Indeed, it follows by direct verification that the tilting modules with the highest weights corresponding to the ones in \ref{indecA2} have the same characters for $SL_3$. We can thus consider the tensor ideal generated by $T_3$ and show that it contains the second tensor power of $\unit\hookrightarrow T_2^L$ (but not $\unit\hookrightarrow T_2^L$ itself due to the cell structure from \cite[Example~15]{An2}). Indeed, an equivalent claim is that the composite $T_2^R\to\unit\to T_2^L$ factors through $T_3$, which can be verified by direct Soergel bimodule calculus due to \cite[Theorem~1.8]{RW}.
\end{proof}

\appendix

\section{Duflo involutions in monoidal categories}\label{sec:Duflo-Mon}

\subsection{Set-up}\label{sec:su}

\subsubsection{}\label{DefS}
We let $\cS$ be a pseudo-tensor category over $\bk$ with a `grading shift functor'. Concretely, this is an auto-equivalence $F:\cS\to\cS$ of $\cS$ as a bimodule category over itself. We will write $F^i(X)=X\langle i\rangle$ for $i\in\mZ$.
We impose the following conditions on this structure:
\begin{enumerate}
\item There exists a set $\mW$ and family of indecomposable objects
$\{B_x\mid x\in \mW\}$  in $\cS$
that is closed under taking duals, and such that for every indecomposable object $X$ in $\cS$ there is precisely one pair $(x,i)\in \mW\times \mZ$ so that $X$ is isomorphic to $B_x\langle i\rangle$.
\item  For each pair $x,y\in \mW$ and $i\le 0$, we have
$$\dim_{\bk}\Hom(B_x,B_y\langle i\rangle)\;=\;\delta_{x,y}\,\delta_{i,0}.$$
\end{enumerate}

\subsubsection{} By the above, there is an element $e\in \mW$ for which $B_e\simeq\unit$. Moreover, there are two mutually inverse operations $x\mapsto x^\ast$ and $x\mapsto {}^\ast x$ on $\mW$ such that $B^\ast_x=B_{x^\ast}$ and ${}^\ast B_x=B_{{}^\ast x}$.

Since $B_x\langle i\rangle\simeq B_x\langle j\rangle \otimes (\unit\langle i-j\rangle)$, the information of thick left (resp. right) tensor ideals and left (resp. right) cells is unambiguously contained in a pre-order $\preceq_L$ (resp. $\preceq_R$) and the corresponding equivalence relation $\sim_L$ (resp. $\sim_R$) on $W$. Concretely, we set $x\preceq_L y$ if and only if $B_y$ is a direct summand of $Y\otimes B_x$ for some $Y\in \cS$, and $B_x$ and $B_y$ belong to the same left cell if and only if $x\sim_L y$ (meaning $x\preceq_L y$ and $y\preceq_L x$). We also speak of the equivalence classes of $\sim_L$ on $\mW$ as left cells.

Based on the case where $\mW$ is a Coxeter group, and $\cS$ the category of Soergel bimodules, see Appendix~\ref{SecSB}, we will label statements according to Lusztig's labelling of Conjectures P1-P15 in \cite[\S 14]{Lubook}.

\subsection{Distinguished elements}\label{sec:disti}

\subsubsection{Lusztig's a-function}\label{DefaD}

For each $x\in \mW$, we set
$$\a(x)\;:=\; \sup\{i\in\mZ\mid B_x\langle i\rangle \mbox{ is a direct summand of } B_y\otimes B_z,\mbox{ for }y,z\in \mW\}\;\in\;\mZ\cup\{\infty\},$$
and
$$\Delta(x)\;:=\; \min\{i\in\mN\mid \Hom(\unit, B_x\langle i\rangle)\not=0\}\;\in\;\mN\cup\{\infty\},$$
with the convention that the minimum of the empty set is $\infty$. Since $B_x$ is a direct summand of $B_x\otimes B_x^\ast\otimes B_x$, the set in the definition of $\a(x)$ is never empty.

\begin{lemma}[P1]\label{LemP1}
We have $\a(x)\le\Delta(x)$ for all $x\in \mW$.
\end{lemma}
\begin{proof}
We can focus on the case $\Delta(x)<\infty$. So assume that there is a non-zero morphism $\unit\to B_x\langle i\rangle$ for some $i\in\mN$, and take $a\in\mZ$ and $y,z\in\mW$ such that $B_x\langle a\rangle$ is a direct summand of $B_y\otimes B_z$. Then
$$0\;\not=\;\Hom(\unit, B_y\otimes B_z\langle i-a\rangle)\;\simeq\; \Hom(B_{y^\ast}, B_z\langle i-a\rangle),$$
which indeed forces $a\le i$ by assumption (2) in \ref{DefS}.
\end{proof}

We are thus led to define a set of `distinguished' or `Duflo' elements in $\mW$.
\begin{definition}
Set 
$\sD\;:=\;\{x\in \mW\mid \a(x)=\Delta(x)<\infty\}\;\subset \;\mW.$
\end{definition}

\begin{remark}
There is a general notion of `Duflo objects' in rigid monoidal categories due to Mazorchuk and Miemietz (more generally, Duflo 1-morphisms in 2-categories), and we show in \S\ref{sec:MM} that P13' below simply states that our notion of Duflo elements, defined in specific categories, actually leads to Duflo objects in the general sense of \cite{MM}.
\end{remark}

\begin{theorem}${}$\label{ThmP3}
\begin{enumerate}
\item[P2.] If $d\in\sD$ and $B_d\langle\a(d)\rangle$ is a direct summand of $B_y\otimes B_z$, then $y={}^\ast z$.
\item[P3.] For each $z\in\mW$, there is a unique $d\in\sD$ such that $B_d\langle\a(d)\rangle$ is a direct summand of ${}^\ast B_z\otimes B_z$.
\item[P5.] If $d\in\sD$ then the multiplicity of $B_d\langle\a(d)\rangle$ in ${}^\ast B_z\otimes B_z$ is at most $1$. Moreover,
$$\dim_{\bk}\Hom(\unit, B_d\langle\a(d)\rangle)\;=\; 1.$$
\item[P13.] Any left cell $\Gamma$ (resp. right cell $\mathrm{P}$) contains a unique $d\in\sD$, and $B_d\langle\a(d)\rangle$ is a direct summand of ${}^\ast B_z\otimes B_z$ for each $z\in \Gamma$ (resp. $B_y\otimes B_y^\ast$ for each $y\in \mathrm{P}$).
\item[P13'.] For a right cell $\mathrm{P}$, the coevaluation $\co_{B_y}$, for each $y\in \mathrm{P}$, factors through the direct summand $B_d\langle\a(d)\rangle$ of $B_y\otimes B_y^\ast$, for $\mathrm{P}\cap\sD=\{d\}$.
\end{enumerate}
\end{theorem}
\begin{proof}
By \ref{DefS}(2), \begin{equation}\label{BB10}\delta_{y^\ast,z}\;=\;\dim \Hom(B_{y^\ast},B_z)\;=\; \dim\Hom(\unit, B_y\otimes B_z).\end{equation}
Hence, if there exists a $B_x\langle\Delta(x)\rangle$, with $\Delta(x)<\infty$, that is a direct summand of $B_y\otimes B_z$, then firstly $y={}^\ast z$, secondly $x\in \sD$ by Lemma~\ref{LemP1}, thirdly $\Hom(\unit,B_x\langle \a(x)\rangle)=\bk$ and fourthly $x$ must be unique. This immediately proves P2, uniqueness in P3 and the first claim in P5. Moreover, by definition, each $B_d\langle\a(d)\rangle$ must occur as a summand in some $B_y\otimes B_z$, proving the second part of P5.

We will actually prove the analogue of P13' for left cells, to be closer to the choices in the other parts. Two indecomposable objects $X$ and $Y$ are in the same left cell if and only if they generate the same left tensor ideal. By Theorem~\ref{ThmRigBij}, this is equivalent to the coevaluations $\unit\to {}^\ast X\otimes X$ and $\unit\to {}^\ast Y\otimes Y$ generating the same subfunctor of $\cS(\unit,-)$.

Fix a left cell $\Gamma$ and take $z\in \Gamma$. By \eqref{BB10} for $y={}^\ast z$, there are unique $x\in\mW$ and $i\in\mZ$ such that the coevaluation of ${}^\ast B_z$ factors via a direct summand $B_x\langle i\rangle$ of ${}^\ast B_z\otimes B_z$. In fact, we must have $i=\Delta(x)$. Indeed, if $\Hom(\unit, B_x\langle j\rangle)\not=0$ with $j<i$, then the same adjunction as used in \eqref{BB10} implies that $\Hom(B_z,B_z\langle j-i\rangle)\not=0$, contradicting \ref{DefS}(2). By the first paragraph $d:=x\in \sD$. This already completes the proof P3 by proving existence.

Now consider another $z'$, in the same left cell $\Gamma$ as $z$. This leads to another $d'\in \sD$. By the second paragraph, $\unit\to B_x\langle \a(x)\rangle$ and $\unit\to B_{d'}\langle \a(d')\rangle$ generate the same subfunctor, which requires in particular non-zero morphisms $B_d\langle \a(d)\rangle\leftrightarrow B_{d'}\langle \a(d')\rangle$, and thus $d=d'$ by \ref{DefS}(2).

Next we show that $d\in \Gamma$. We already know that $B_d\langle \a(d)\rangle$ is a direct summand of ${}^\ast B_z\otimes B_z$ for any $z\in \Gamma$ (so in particular $z\preceq_L d$), and that coevaluation $\unit\to {}^\ast B_z\otimes B_z$ factors through this summand. Now we can use that the composition
$$B_z\xrightarrow{ B_z\otimes \co}B_z\otimes {}^\ast B_z\otimes B_z\xrightarrow{\ev \otimes B_z}B_z $$
is the identity on $B_z$ to conclude that $B_z$ is a direct summand of $B_z\otimes B_d\langle\a(d)\rangle$ (so $d\preceq_L z$), and hence $B_z$ and $B_d$ are in the same left cell.

To conclude the proof it only remains to be observed that each left cell contains only one element of $\sD$. But this is now immediate, since each $B_d\langle\a(d)\rangle$ must be a direct summand of some ${}^\ast B_z\otimes B_z$, but by uniqueness in P5, this must be the $d$ in the left cell of $z$ studied in the previous three paragraphs.
\end{proof}

%
%
%
%
%
%
%
%

\begin{remark}
We can weaken our assumptions on $\cS$ to demand only existence of one-sided duals, leading to the same results as above, but only in relation to left or right cells.
\end{remark}

\subsection{Duflo involutions}\label{sec:invo}

Even to state the property that $\sD$ consists of `involutions', we require additional structure on $\cS$.

\subsubsection{}\label{Piv} From now on we additionally require that $X^{\ast\ast}\simeq X$ for every $X\in\cS$ (for instance, $\cS$ has a pivotal structure). Then the operations $x\mapsto x^\ast$ and $x\mapsto {}^\ast x$ are identical, and we denote the operation by $x\mapsto x^{-1}$, as it is involutory.
We additionally assume that there is a monoidal equivalence $I:\cS\to\cS^{\op},X\mapsto \overline{X}$ with $\overline{B_x}=B_x$ and $\overline{X\langle 1\rangle}=\overline{X}\langle -1\rangle$. By $\cS^{\op}$ we denote the category with morphisms reversed (but not the tensor product).

\begin{theorem}[P6]\label{ThmP6} Under the assumptions in \ref{Piv},
each element $d\in\sD$ satisfies $d=d^{-1}$.
\end{theorem}
\begin{proof}
The existence of $I$ shows that $\Delta(x)=\Delta(x^\ast)$ and $\a(x)=\a(x^\ast)$ for all $x\in\mW$, in particular $\cD$ is closed under $x\mapsto x^\ast$.
If $B_d\langle \a(d)\rangle$ is a direct summand of $B_y\otimes B_y^\ast$, then $B_d^\ast\langle -\a(d)\rangle$ is a direct summand of $B_y^{\ast\ast}\otimes B_y^\ast$. The combination of both additional assumptions thus shows that $B_{d^{-1}}\langle \a(d)\rangle=B_{d^{-1}}\langle \a(d^{-1})\rangle$ is a direct summand of $B_y\otimes B_y^\ast$. Since $d^{-1}\in\sD$, uniqueness in P5 of Theorem~\ref{ThmP3} shows that $d=d^{-1}$.
\end{proof}

%

\begin{definition}\label{Def:dist}
A non-zero morphism $\unit\to B_x\langle i\rangle$, for $x\in\mW$ and $i\in\mZ$, is {\bf distinguished} if $x\in\sD$ and $i=\a(x)$.
\end{definition}
By P5 in Theorem~\ref{ThmP3}, distinguished morphisms are unique up to scalar.
 We have the following criterion.

\begin{lemma}\label{LemAlt}
Under the assumptions in \ref{Piv}, the following conditions are equivalent on a non-zero morphism $\alpha:\unit\to B_x\langle i\rangle$.
\begin{enumerate}
\item $\alpha$ is distinguished;
\item $\alpha=\beta\circ (\alpha^{\otimes 2})$ for some morphism $\beta$ in $\cS$;
\item $\alpha=\pi\circ (\alpha^{\otimes 2})$ for a projection $\pi$ from $B_x^{\otimes 2}\langle 2i\rangle$ onto a direct summand $B_x\langle i\rangle$.
\end{enumerate}
\end{lemma}
\begin{proof}
That (3) implies (2) is trivial.

Assume that (2) is satisfied for some morphism $\beta$ in $\cS$. Then there needs to exist at least one summand $B_y\langle j\rangle$ of $B_x^{\otimes 2}\langle 2i\rangle$ with non-zero morphisms $\unit\to B_y\langle j\rangle\to B_x\langle i\rangle$. In particular $j\le i$. By applying $I$ and shifting by $2j$, we find that $B_{y}\langle j\rangle$ is also a direct summand of $B_x^{\otimes 2}\langle 2j-2i\rangle$. It follows that there is a non-zero morphism $\unit\to B_x^{\otimes 2}\langle 2j-2i\rangle$, so by adjunction a non-zero morphism $B_x^\ast\to B_x\langle 2j-2i\rangle$. This requires $j=i$, $x=x^{-1}$ and $x=y$.
From the start we had $\Delta(x)\le i $, but since $B_x\langle i\rangle$ is now a direct summand of $B_x^{\otimes 2}$, we also have $i\le \a(x)$. Lemma~\ref{LemP1} thus shows that $x\in\sD$ and $i=\a(x)$, implying that $\alpha$ is distinguished, so (1) follows.

Finally, assume that $\alpha:\unit\to B_d\langle \a(d)\rangle$ is distinguished. Since coevaluation $\unit\to B_d\otimes B_d$ factors via $\alpha$, it follows (as already observed in the proof of Theorem~\ref{ThmP3}) that $B_d\otimes\alpha$ is a split monomorphism $B_d\hookrightarrow B_d^{\otimes}\langle \a(d)\rangle$, and we let $\pi$ be the splitting map. It then follows that $\alpha=\pi\circ\alpha^{\otimes 2}$, so we conclude that (1) implies (3).
\end{proof}

\begin{remark}We can rephrase Lemma~\ref{LemAlt} as saying that $\alpha$ is distinguished if and only if $\alpha$ and $\alpha^{\otimes 2}$ generate the same tensor ideal (where `tensor ideal' can freely be interpreted as left, right or two-sided).
\end{remark}

\subsection{Duflo involutions for Coxeter groups without using boundedness}\label{SecSB}

Let $(W,S)$ be a Coxeter system of finite rank ($S$ is finite). In \cite{Lubook}, Lusztig proved properties P1-P15 concerning Duflo elements in the {\em equal} parameter case, under the assumption that the system is `bounded', as in \cite[\S 13.2]{Lubook}. A proof of this boundedness has recently been announced in \cite{bounded}, thus in particular completing Lusztig's proof. Our methods give some results that do not rely on the boundedness hypothesis:

\begin{theorem}\label{Thm:unbounded}
In the equal parameter case, properties P1, P2, P3, P5, P6 and P13 in \cite[Conjectures~14.2]{Lubook} of $(W,S)$ are true.
\end{theorem}

To prove the theorem we let $\bk$ be a field of characteristic 0.

\subsubsection{}\label{SBM}We choose a reflection faithful representation of $W$ on a finite dimensional $\bk$-vector space $V$. An example is given in \cite[Proposition~2.1]{So} for $\bk=\mR$. Then one can associate the corresponding category $\cS_W$ of Soergel   from \cite{So}. These are specific graded bimodules over the graded $\bk$-algebra $R:=\Sym(V)$.

Then $\cS_W$ is by construction a $\bk$-linear additive idempotent complete monoidal category with grading shift functor $\langle1\rangle$, generated by bimodules $B_s$, $s\in S$. It is well-known, see for instance \cite{EK}, that $B_s$ is monoidally self-dual as a graded $R$-bimodule, so that by construction $\cS_W$ is rigid and therefore a pseudo-tensor category.

By \cite[Satz 6.14(2)]{So}, the indecomposable objects in $\cS_W$ are labelled as $B_w\langle i\rangle$ for $w\in W$ and $i\in\mZ$. It further follows from \cite[Lemma~6.13]{So} and \cite[Theorem~1.1]{EW14} that
\begin{equation}\label{eqgrrank}\sum_{i\in\mZ}\dim\Hom(B_x, B_y\langle i\rangle)v^i\;=\;\left(\sum_{j\in\mN}\dim R_j v^j \right)\sum_{z\in W}h_{z,x}h_{z,y}.\end{equation}
Here $h_{z,x}\in\mZ[v]$ are the Kazhdan-Lusztig polynomials from \ref{ordKL}, so that $h_{x,x}=1$ and $h_{z,x}$ has no constant term if $z\not=x$. It follows that the conditions in \ref{DefS} are satisfied.

Via its definition of $B_w$ as a unique summand of a tensor product of $B_s$, with $s$ running over a reduced expression of $w$, it then follows that naturally $B_w^\ast\simeq {}^\ast B_w\simeq B_{w^{-1}}$.
A functor $I$ as in \ref{Piv} is given by monoidal duality followed by exchanging left and right $R$-actions, so the conditions in \ref{Piv} are also satisfied.

\begin{proof}[Proof of Theorem~\ref{Thm:unbounded}]
By the discussion in \ref{SBM}, this follows from Lemma~\ref{LemP1} and Theorems~\ref{ThmP3} and~\ref{ThmP6}, if we show that the functions $\a$ and $\Delta$ as in \ref{DefaD} applied to $\cS_W$ produce the same functions as those defined from $W$ in \cite[\S 13 and \S 14]{Lubook}. For $\a$, this follows from the fact that the split Grothendieck ring of $\cS_W$ is isomorphic to the Hecke algebra where the basis $[B_x\langle i\rangle]$ is sent to $v^i\underline{H}_x$, see \cite[Theorem~1.1]{EW14}. For $\Delta$ this follows from \eqref{eqgrrank} which shows that the minimal $i$ for which $\Hom(\unit,B_x\langle i\rangle)$ is non-zero is the maximal $i$ for which $v^{-i}h_{e,x}\in\mZ[v]$.
\end{proof}

\begin{remark}
Lusztig's conjectures in the {\em unequal} parameter case remain an active topic, see for instance \cite{GP}.
\end{remark}

\subsection{Mazorchuk - Miemietz - Duflo objects}\label{sec:MM}
The following theorem (up to formulation) appears as \cite[Proposition~17]{MM}, and the resulting objects are subsequently referred to as `Duflo' objects in \cite{MM2}, in case $\cA$ is a pseudo-tensor category with finitely many isomorphism classes of indecomposable objects, but without the `$\End(\unit)=\bk$'-condition. Actually, the result was proved for a type of 2-categories generalising these monoidal categories.
 As is known by the authors of \cite{MM}, the result does not require the condition of finitely many indecomposables. Theorem~\ref{ThmMM}(1) below is a reformulation that avoids the usage of the monoidal (via Day convolution) module category of $\cA$. We do lose some of the elegance in \cite{MM} by restricting to the language in the rest of the current paper.

 As in \S \ref{sec:su}, by a {\em right cell} we refer to a  non-empty collection of indecomposable objects $P$ so that if $X\in P$ then $Y\in P$ if and only if there are $U,V\in\cA$ with $X$ a direct summand of $Y\otimes U$ and $Y$ a direct sum of $X\otimes V$. Hence right cells are in natural bijection with right thick tensor ideals generated by a non-zero indecomposable object.

\begin{theorem}\label{ThmMM}
Let $\cA$ be an essentially small pseudo-abelian Krull-Schmidt rigid monoidal category with additive tensor product (e.g. a pseudo-tensor category) and $P$ a right cell. There exists a unique, up to post-composition with an isomorphism, morphism $\delta:\unit\to D$ in $\cA$ with $D$ indecomposable, characterised by {\bf either} of the following three properties:
\begin{enumerate}
\item For every $Y\in P$, the morphism $Y\xrightarrow{\delta\otimes Y} D\otimes Y$ is a split monomorphism, and $D$ is in the left thick tensor ideal generated by $ Y^\ast$.
\item For every $Y\in\cA$, we have $Y\in P$ if and only if $\co_Y$ factors through $\delta$ via a direct summand $D$ of $Y\otimes Y^\ast$.
\item $\delta$ generates the minimal right tensor ideal containing $\id_Y$ for $Y\in P$.
\end{enumerate}
Moreover, $D$ belongs to the right cell $P$.
\end{theorem}
\begin{proof}
Let $(D,\delta)$ and $(D',\delta')$ both be as in (3). Then by Theorem~\ref{ThmRigBij}, which holds in the present context, see \cite{Selecta}, there are morphisms $f:D\to D'$ and $g:D'\to D$ such that $\delta'=f\circ \delta$ and $\delta=g\circ \delta'$. If $f$ is not an isomorphism, then $g\circ f$ belongs to the Jacobson radical $J$ of $R:=\End(D)$, so that for the $R$-submodule $N$ of $\Hom(\unit,D)$ generated by $\delta$, we have $JN=N$ by $g\circ f\circ\delta=\delta$, contradicting Nakayama's Lemma. Hence $(D,\delta)$ as in (3) are unique in the desired sense.

Also for $(D,\delta)$ as in (3), it is clear, by Theorem~\ref{ThmRigBij}, that for $Y\in P$ the coevaluation $\co_Y$ factors through $\delta$. If it would not factor via a direct summand $D$ in $Y\otimes Y^\ast$, then $\co_Y$ would generate, by precisely the same argument as in the previous paragraph, a strictly smaller subfunctor of $\cA(\unit,-)$, and hence a smaller right tensor ideal $\cI$, which is a contradiction as $\cI$ still contains $P$. Conversely, consider $Y\not\in P$. If $Y$ is not in the thick right tensor ideal generated by $P$ then $\co_Y$ cannot factor via $\delta$. If $Y$ is in the thick right tensor ideal generated by $P$ (but not in $P$), then $\co_Y$ cannot factor through $\delta$ via a direct summand $D$, since otherwise $Y$ would generate a right tensor ideal including $P$. Hence (3) implies (2).

That (2) implies (1) follows directly by adjunction. 
Next we prove that (1) implies (3).
Consider $(D,\delta)$ as in (1). For $Y\in P$ it follows via adjunction that $\co_Y$ factors via $\delta$. We thus only need to show that $\delta$ does not generate a strictly bigger right tensor ideal. But since $D$ is a direct summand of $Z\otimes Y^\ast$ for some $Z\in\cA$, the morphism $\delta$ factors as
$$\unit\xrightarrow{\alpha} Z\otimes Y^\ast\to D.$$
By further writing $\alpha$ as a composite of $\co_{Y}$ with some $f\otimes Y^\ast$ it follows that $\delta$ is contained in the minimal ideal from (3), concluding the proof of (1)$\Rightarrow$(3).

Next, we prove existence of a $(D,\delta)$ as in (3). Let $\delta_i:\unit\to D_i$ be a collection of generators of the subfunctor of $\cA(\unit,-)$ corresponding to the ideal in (3), via Theorem~\ref{ThmRigBij}. Firstly we can observe that the collection can be taken to be finite, since $\co_Y$, for any $Y\in P$, generates the subfunctor. Then, by principle \eqref{eq:sumideals} (which shows that the tensor ideal generated by at least one of the morphisms must contain $\id_Y$) we can conclude that a singleton must suffice. .

Finally, clearly $D$ as in (2) is in the right tensor ideal generated by $Y$. The converse follows precisely as in the penultimate paragraph of the proof of Theorem~\ref{ThmP3}.
\end{proof}

\begin{remark}
With the approach in \cite[Proposition~17]{MM} one can show that also ${}^\ast D$ is in $P$, but interestingly in general $D^\ast\not\in P$.
\end{remark}

\begin{remark}
We have the same statement as in Theorem~\ref{ThmMM} for a left cell $\Gamma$. The first two equivalent characterisations become:
\begin{enumerate}
\item For every $Y\in\Gamma$ the morphism $Y\xrightarrow{Y\otimes\delta} Y\otimes D$ is a split monomorphism and $D$ is in the right thick tensor ideal generated by ${}^\ast Y$.
\item For every $Y\in\cA$, we have $Y\in \Gamma$ if and only if $\co_{{}^\ast Y}$ factors through $\delta$ via a direct summand $D$ of ${}^\ast Y\otimes Y$.
\end{enumerate}
We can also formulate a dual version. Let $P$ be a right cell. There exists unique (up to isomorphism) $\delta: D\to\unit$ with $D$ indecomposable that can be characterised by  either of:
\begin{enumerate}
\item For every $Y\in P$, the morphism $D\otimes Y\xrightarrow{\delta\otimes Y}Y$ is a split epimorphism and $D$ is in the left thick tensor ideal generated by ${}^\ast Y$.
\item For every $Y\in\cA$, we have $Y\in P$ if and only if $\ev_{{}^\ast Y}$ factors through $\delta$ via a direct summand $D$ of $Y\otimes {}^\ast Y$.
\end{enumerate}
One can pass from one version of the theorem to another by considering the opposite/reversed monoidal categories.
\end{remark}

With Theorem~\ref{ThmMM} stated in sufficient generality, and with the results of Theorem~\ref{ThmP3} and \S \ref{SBM} in mind, it is now tempting to extend the concept of Duflo elements to the $p$-version of Kazhdan-Lusztig theory, see \cite{JW}. Crucially, in this case the categorification does not satisfy~\ref{DefS}(2) so that the approach of Sections \ref{sec:disti} and \ref{sec:invo} does not apply.
\begin{example}
Let $(W,S)$ be a Coxeter group and $\bk$ a field of characteristic $p>0$. Theorem~\ref{ThmMM} applies to the diagrammatic category of Soergel bimodules $\cS_W$ from~\cite{EW16}, a pseudo-tensor category over $\bk$. The right cells in $\cS_W$ are in canonical bijection with the $p$-Kazhdan-Lusztig cells in $W$ and, for each right $p$-cell $P\subset W$, Theorem~\ref{ThmMM} defines a corresponding ``$p$-Duflo element'' $x\in P$, so that $B_x\langle a\rangle=D$, for some $a\in\mZ$.

For example, if $W$ is the infinite dihedral group with generators $s,t$, then the $p$-Duflo elements are the involutions
$$(st)^{p^n-1}s\quad\mbox{and}\quad (ts)^{p^{n}-1}t,\qquad\mbox{with }\;n\in\mN.$$

In general, the $p$-Duflo elements will be involutions, which can again be deduced from the fact that $\cS_W$ satisfies \ref{Piv}.
\end{example}

\subsection*{Acknowledgements}
The authors thank Jonathan Gruber, Bruno Kahn, Volodymyr Mazorchuk, Daniel Nakano, Joseph Newton, James Parkinson and Geordie Williamson for interesting discussions, and the referee for many useful suggestions resulting in improved presentation. The research of KC was partially supported by ARC grant FT220100125. The work of PE was partially supported by the NSF grants DMS - 2001318 and DMS - 2502467.

\end{document}